\documentclass[11pt, english,reqno]{amsart}
\usepackage[margin=1.1in]{geometry}
\usepackage{amsmath,amsthm,amssymb,mathtools}
\usepackage[T1]{fontenc}
\usepackage[utf8]{inputenc} 
\usepackage{microtype}
\usepackage{enumitem}
\usepackage{xcolor}
\usepackage{hyperref}
\usepackage[nameinlink,noabbrev]{cleveref}
\usepackage{esint}
\hypersetup{
  colorlinks=true,
  linkcolor=blue,
  citecolor=blue,
  urlcolor=blue
}

\numberwithin{equation}{section}

\theoremstyle{plain}
\newtheorem{theorem}{Theorem}[section]
\newtheorem{lemma}[theorem]{Lemma}
\newtheorem{proposition}[theorem]{Proposition}
\newtheorem{corollary}[theorem]{Corollary}

\theoremstyle{definition}
\newtheorem{definition}[theorem]{Definition}

\theoremstyle{remark}
\newtheorem{remark}[theorem]{Remark}

\newcommand{\R}{\mathbb{R}}
\newcommand{\N}{\mathbb{N}}

\newcommand{\Sph}{\mathbb{S}}

\newcommand{\eps}{\varepsilon}

\newcommand{\cA}{\mathcal{A}}

\DeclareMathOperator{\FB}{FB}

\author{Xavier Fern\'andez-Real}
\address{EPFL SB, Station 8, 1015 Lausanne, Switzerland}
\email{xavier.fernandez-real@epfl.ch}

\author{Enric Florit-Simon}
\address{ETH Zurich, R\"amistrasse 101, 8092 Zurich,
Switzerland}
\email{enric.florit@math.ethz.ch}

\author{Joaquim Serra}
\address{ETH Zurich, R\"amistrasse 101, 8092 Zurich,
Switzerland}
\email{joaquim.serra@math.ethz.ch}

\title[Finite index solutions to the Bernoulli problem]{Finite index solutions to the Bernoulli problem\\ in three dimensions are axially symmetric}

\thanks{X. F. was supported by the Swiss National Science Foundation (SNF grant PZ00P2\_208930), by
the Swiss State Secretariat for Education, Research and Innovation (SERI) under contract number
MB22.00034, and by the AEI project PID2021-125021NA-I00 (Spain). E. F. and J. S. were supported by the European Research Council under the Grant Agreement No 948029.}


\date{} 

\begin{document}

\begin{abstract}
We show that every entire solution to the Bernoulli (or one-phase) free boundary problem with finite Morse index in $\R^3$ is axially symmetric.  In fact, we additionally prove that the same result would follow in any dimension $4 \le n \le 6$ in which stable entire solutions are shown to be flat.
\end{abstract}

\maketitle

\tableofcontents

\section{Introduction}
Given an open set $D\subset \R^n$, we say that $u: D \to [0, \infty)$ is a classical solution to the Bernoulli (or one-phase) free boundary problem if:
\begin{equation}
\label{eq:classical_Bernoulli}
    \text{$\{u > 0\}$ is locally a smooth domain in $D$}\quad\text{and}\quad 
    \left\{\begin{array}{rcll}
    \Delta u & = & 0 &  \quad \text{in}\quad \{u > 0\}\cap D,\\
    |\nabla u| & = & 1 & \quad\text{on}\quad \partial \{u > 0\}\cap D.
    \end{array}
    \right.
\end{equation}
Such solutions are critical points of the \emph{Alt-Caffarelli functional}
\[
\mathcal{E}(v; A) := \int_A \left( |\nabla v|^2 + \, \mathbf{1}_{\{v >0\}} \right)\,  dx
\]
in any open subset $A$ with compact closure in $D$. This functional was originally introduced by Alt and Caffarelli in \cite{altcaffarelli1981onephase}, and since then, it has become one of the most studied free boundary problems (see also the monographs \cite{CaffarelliSalsa2005, Vel23}). In recent years, the one-phase Bernoulli problem has been studied increasingly beyond the classical minimizing framework, \cite{kamburov2022nondegeneracy, EFY23, BK24, KW24,  CFFS25, EngelsteinRestrepoZhao2025, HinesKolesarMcGrath2025}. 

We say that  $u$ is  \emph{stable in $A$} if the second variation quadratic form of the Alt--Caffarelli functional is nonnegative. That is, 
\begin{equation}
\label{eq:Q_nonnegative}
\mathcal{Q}(\varphi, \varphi; A) := \int_{A\cap \{u > 0\}} |\nabla \varphi|^2\,dx - \int_{A\cap \FB(u)} H\,\varphi^2\, d\mathcal{H}^{n-1}
\ge 0 \qquad \forall \ \varphi \in C_c^\infty\!\bigl(A\cap \overline{\{u > 0\}}\bigr),
\end{equation}
where $H$ denotes the mean curvature of the \emph{free boundary} $\FB(u) := \partial\{u > 0\}$ with respect to the outer unit normal $\nu = -\nabla u$. The definition extends to the case where $A$ is noncompact, since we always consider compactly supported variations.

More generally, $u$ is said to have \emph{Morse index $m\in\N\cup \{0\}$ in $A\subset D$} if
\begin{align*}
    m=\sup\Big\{&\dim E : E\subset C_c^\infty\!\bigl(A\cap \overline{\{u > 0\}}\bigr)\ \mbox{is a linear subspace with}\ 
\mathcal Q(\varphi,\varphi; A)<0\  \forall \varphi\in E\setminus\{0\}\Big\}.
\end{align*}
If $D=\R^n$, then the finiteness of index in $\R^n$ implies stability in $\R^n\setminus \overline{B_{R_0}}$ for some $R_0>0$; see \cite[Proposition 3.4]{BK24}.\\

The goal of this article is to characterize global classical solutions to the Bernoulli problem with finite Morse index in three dimensions, and more generally in every dimension $3\leq n \leq 6$ in which stable entire solutions are known to be flat. This had been previously achieved only in the two-dimensional case \cite{BK24, Tra14}, exploiting the log cut-off trick and a deep, one-to-one correspondence between entire solutions to the Bernoulli problem with finite connectivity in the plane and minimal bigraphs in $\R^3$---both of which are not available in higher dimensions.

We emphasise that there is a wealth of entire classical solutions to \eqref{eq:classical_Bernoulli} without further constraints, as showcased by the examples in \cite{BakerSaffmanSheffield1976, HHP11, Tra14, LWW21, DJS22, ESV24} and their natural extensions to additional dimensions. They have rich structures in terms of free boundary connectivity, topology and symmetries; in dimension $n\geq 7$, the examples from \cite{DJS22, ESV24} are even energy-minimizing.

The finite index hypothesis and dimensional restrictions are then essential in our result. They will lead to a decomposition of such solutions---away from a large ball---into a sum of two solutions with connected free boundaries that become progressively flatter at infinity (see Section~\ref{sec:asymptind}).  We will then combine a novel improvement of flatness in annuli technique developed in Section~\ref{sec:improvflat} (see also the forthcoming work \cite{FFS25A}) with the moving planes method \`a la Schoen/Serrin \cite{Schoen83, Serrin71} in Section~\ref{sec:symindex} and several new arguments particular to the Bernoulli problem, leading to the axial symmetry of all such solutions. See Section~\ref{sec:mainresults} for more details.\\ 

The study of finite index solutions to geometric variational problems, and in particular of entire solutions in Euclidean space, has a long history in the literature. Such solutions provide many of the known examples \cite{Costa84, dPKW13}. Moreover, they are especially relevant to the existence theory in compact geometries, see \cite{MarquesNeves2016, GasparGuaraco2018, CM20}. For further examples of structure and classification results for solutions of finite index to other geometric variational problems, we refer the reader to \cite{FC85, Tysk89, Schoen83, CL24, CLMS25, Maz24} for minimal hypersurface theory, \cite{WangWei18, GuiLiuWei,GuiWangWei2020, KLP12, Flo25} for the Allen--Cahn equation, \cite{wang2015structure, DGW22} for its elegant free boundary version, and references therein.

\subsection{Main results}\label{sec:mainresults}
Our main result in dimension 3 is the following:

\begin{theorem}\label{thm:axymR3}
    Let $u:\R^3\to[0,\infty)$ be an entire classical solution to the Bernoulli problem with finite Morse index which is not one-dimensional. Then, after a suitable scaling, translation, and rotation, we have that:
    \begin{itemize}
        \item either $u(x) = (1-|x|^{-1})_+$;
        \item or  $u(x_1, x_2, x_3) = u(r, x_3)$ is axially symmetric, with $ r= \sqrt{x_1^2+x_2^2}$, even in $x_3$, and satisfies
        \[
        \partial_{x_3} u > 0\quad\text{in}\quad \{x_3 > 0\}\cap \{u > 0\},\qquad \partial_{r} u < 0\quad\text{in}\quad \{r > 0\} \cap \{u > 0\}.
        \]
        Moreover, $\{u = 0\} = \{(x', x_3) : |x'|\ge 1, |x_3| \le g(|x'|)\}$ for some $g:[1,\infty) \to [0, \infty)$ with $g(1) = 0$, $g'(r) > 0$ for $r >1$, and $\lim_{r\to\infty} g(r) < \infty$. 
    \end{itemize}
\end{theorem}
In fact, we obtain a sharp conditional result up to dimension $6$.
Define the \textit{critical dimension}
\begin{equation}\label{eq:nstardef}
    n_* := \min\{n_*^1,n_*^2\},
\end{equation}
where
\begin{equation}\label{eq:nstardef1}
    n_*^1 := \min\left\{n \in \N : \ \begin{array}{l}\text{there exists a classical stable solution}\\
    \text{in $\R^n$ which is not one-dimensional}\end{array}\right\},
\end{equation}
\begin{equation}\label{eq:nstardef2}
    n_*^2 := \min\left\{n \in \N : \ \begin{array}{l}\text{there exists a classical stable one-homogeneous}\\
    \text{solution in $\R^n\setminus\{0\}$ which is not one-dimensional}\end{array}\right\}.
\end{equation}
By \cite{CJK04,jerisonsavin2015cones,DeSilvaJerison2009} and \cite{CFFS25, DJS22,ESV24} we have $5\leq n_*^1\leq 7$ and $4\leq n_*^2\leq 7$, hence
\[
4 \le n_* \le 7. 
\]
Then, an analogous result holds for $3 \le n < n_*$:
\begin{theorem}\label{thm:axymRn}
    Let $3\leq n < n_*$, with $n_*\ge 4$ given by \eqref{eq:nstardef}. Let $u:\R^n\to[0,\infty)$ be an entire classical solution with finite Morse index which is not one-dimensional. Then, after a suitable scaling, translation, and rotation, we have that:
    \begin{itemize}
        \item either $u(x)=\frac{1}{n-2}(1-|x|^{2-n})_+$;
        
        \item or  $u(x_1,\dots,x_{n-1}, x_n) = u(r, x_n)$ is axially symmetric, with $ r= \sqrt{x_1^2+\dots + x_{n-1}^2}$, even in $x_n$, and satisfies
        \[
        \partial_{x_n} u > 0\quad\text{in}\quad \{x_n > 0\}\cap \{u > 0\},\qquad \partial_{r} u < 0\quad\text{in}\quad \{r > 0\} \cap \{u > 0\}.
        \]
        Moreover, $\{u = 0\} = \{(x', x_n) : |x'|\ge 1, |x_n| \le g(|x'|)\}$ for some $g:[1,\infty) \to [0, \infty)$ with $g(1) = 0$, $g'(r) > 0$ for $r >1$, and $\lim_{r\to\infty} g(r) < \infty$. 
    \end{itemize}
\end{theorem}
\begin{remark}
\label{rem:flat_away}
It is interesting to notice that, in fact, our proof shows that any classical entire solution $u:\R^n\to[0,\infty)$, decomposing as $u=u_++u_-$ away from a large ball $B_{R_0}$, with $u_\pm$ asymptotically flat classical solutions in $\R^n\setminus B_{R_0}$, satisfies the thesis of Theorem~\ref{thm:axymRn}.
\end{remark}
\begin{remark}\label{rem:asympintro}
    In fact, we have, for any $\alpha \in (0, 1)$, 
    \[
    \left|u-|x_n|+  b  -\frac{c} {|x|^{n-2}} - \frac{(n-2) b c |x_n|}{|x|^{n}}\right| \le \frac{C_\alpha}{|x|^{n-1+\alpha}}\quad\text{in}\quad \{u > 0\},\]
    for some $b, c > 0$. In particular, the asymptotic expansion for $g:[1, \infty)\to [0, \infty)$ is: 
    \[
    \begin{split}
    g(r) & = b - c r^{2-n} + O(r^{-n+1-\alpha}),\\
    g'(r)& =(n-2)cr^{1-n}+ O(r^{-n-\alpha}),\\
    g''(r)& =-(n-1)(n-2)cr^{-n}+ O(r^{-n-\alpha-1}),
    \end{split}
    \]
    Moreover, we know $g'(r)  > 0$ for $r > 1$, and from the sign on the mean curvature of the free boundary we also know $g''(r) \le - g'(r) (1+g'(r)^2)/r  < 0$ for $r > 1$. 
    \end{remark}
    \begin{remark}
    In \cite{LWW21}, the authors construct an example of an entire axially symmetric  solution in all dimensions $n\geq 3$ via gluing methods. From its symmetries and monotonicity relations, it readily follows that it is stable outside of a ball\footnote{We emphasize that this is known to be equivalent to having finite Morse index for a large class of problems and operators, see \cite{Devyver2012}.}, and our method applies (alternatively, one can use Theorem~\ref{thm:highordexp_intro} below to get their expansion). Furthermore, they show that $\lim_{r\to \infty} g'(r) r^{n-2}\ge 0$ for their solution, while from the expansions above one deduces that, in fact, $\lim_{r\to \infty} g'(r) r^{n-1} \in (0, +\infty)$ in their construction.
\end{remark}
A main step in the proof, which we believe may be of independent interest, is the study of asymptotically flat solutions in exterior domains. Our approach follows an improvement of flatness-type scheme in which the natural domains are annuli (instead of balls, as is the usual case in improvement of flatness-type iterations). This is a simple yet powerful and flexible tool, and which can be applied to other similar settings, see \cite{FFS25A,Flo25}, and compare also with the variational approach of \cite{EdelenSpolaorVelichkov2024}.

We say that a (nonzero) classical solution $u: \R^n \setminus B_{R_0}\to[0,\infty)$ is \textit{asymptotically flat} if for all $R>R_0$ large there exists $e_R\in \mathbb S^{n-1}$ such that
\[
 u =  e_R\cdot x  + o(R) \qquad \mbox{in} \quad  (B_R \setminus B_{R/2}) \cap  \{u>0\}
\]
as $R\to \infty$. The main result in Section~\ref{sec:improvflat} is:
\begin{theorem}\label{thm:highordexp_intro}
Let $\alpha\in (0, 1)$, $R_0>0$, and $u$ be a classical solution to the Bernoulli problem, \eqref{eq:classical_Bernoulli}, in $\R^n \setminus \overline{B_{R_0}}\subset \R^n$, which is asymptotically flat. Then, after a rotation and for some  $(b, c, d) \in \R\times \R\times \R^n$ with ${d}_n=(n-2){ b}{ c}$, we have
\begin{equation}\label{expansion222intro}
         u(x) = x_n- b+\frac{c}{|x|^{n-2}} + \frac{d\cdot x}{|x|^{n}} +  O\left( {|x|^{1-n-\alpha}}\right)\qquad \text{in}\quad \{u > 0\}\setminus B_{R_0},
    \end{equation}
as $x\to \infty$.
\end{theorem} 
Section~\ref{sec:asymptind} shows, via a blow-down analysis, that any entire solution $u:\R^n\to[0,\infty)$ with finite index and $3\leq n <n_*$ decomposes as $u=u_++u_-$ outside of a large ball, where each of the $u_\pm$ is an asymptotically flat classical solution. Theorem~\ref{thm:highordexp_intro} then applies to each of the $u_\pm$, yielding a precise asymptotic expansion for $u$. The resulting asymptotics are then used in a crucial (and subtle) way in Section~\ref{sec:symindex}, in combination with the moving planes method \`a la Schoen/Serrin \cite{Schoen83, Serrin71} and global information coming from stationarity, to deduce Theorem~\ref{thm:axymRn}.

\begin{remark}
    It is interesting to consider which alternative conditions might also provide such a strong characterization of global solutions. In two dimensions, the natural condition of finite connectivity of the free boundary suffices, by \cite{Tra14}. In higher dimensions, however, even finite topology is far from sufficient anymore, by the examples mentioned above. In fact, one expects a wealth of examples in $\R^3$ with finite topology and no elementary characterization, even with a density condition on the positivity set, as suggested by the homogeneous solutions recently constructed in \cite{HinesKolesarMcGrath2025}. In Section~\ref{sec:totHcurv} we propose a new, related condition: the finiteness of the total mean curvature of the free boundaries, defined in a scale-invariant integral sense (see Definition~\ref{def:totHcurv}). It implies finite connectivity in two dimensions, as shown in \cite{BK24}. We will show that it ensures the axial symmetry of global solutions in all higher dimensions too, via the methods developed in the article.
\end{remark}

\section{Solutions in exterior domains}\label{sec:improvflat}

In this section, we study asymptotically flat solutions in exterior domains. 
We first prove an annular improvement-of-flatness result and then use it to derive the higher-order expansion at infinity.

Throughout the paper, whenever \(x=(x',x_n)\in\R^{n-1}\times\R\), we write
\[
B'_r(y'):=\{z'\in\R^{n-1}: |z'-y'|<r\},\qquad B'_r:=B'_r(0).
\]
Moreover, given sets $A_1, A_2, A_3\subset \R^n$, we say 
\[
A_1\subset A_2\quad\text{in}\quad A_3\qquad\Longleftrightarrow \qquad A_1\cap A_3\subset A_2\cap A_3. 
\]

\subsection{Improvement of flatness in annuli}\label{sec:improvflat2}

The following quantity records the best trapping of $u$ between two parallel shifted half-plane solutions on a fixed annulus.

\begin{definition}\label{defi:height} Let $n\ge 3$. For a given classical solution \eqref{eq:classical_Bernoulli} in  a domain $D\subset \R^n$  containing the annulus $B_{2r}\setminus \overline{B_{r/2}}$, we define the \emph{shifted annular height} in the direction $e\in \mathbb S^{n-1}$, with shift $b\in \R$, and at scale $r>0$ as:
\[
    H_b(u,e,r) := \inf\big\{  h>0 \  : \ (e\cdot x -b -h)_+\le u\le (e\cdot x -b+h)_+ \ \mbox{ in } B_{2r}\setminus \overline{B_{r/2}} \big\}.
\]
The \emph{centered annular height} is the particular case $b=0$, $H_0(u, e, r)$. We also define 
\[H(u,r) := \inf_{e\in \mathbb{S}^{n-1},\, b \in \R} H_b(u,e, r)\qquad \mbox{and}\qquad H_0(u,r) := \inf_{e\in \mathbb{S}^{n-1}} H_0(u,e, r).\]
\end{definition}
\medskip

Notice that we have the scaling
\begin{equation}
    \label{eq:scaling}
    H_b(u, e, r) = \rho H_{b/\rho} (u_\rho, e, r/\rho),\qquad\text{for}\quad \rho  >0 \quad \text{and} \quad u_\rho(x) := \frac{u(\rho x)}{\rho}. 
\end{equation}

We next state the improvement of flatness argument in annuli. This is one of the main contributions of the work, and which can be extrapolated to other settings: 

\begin{theorem}\label{thm1}
Let $n\ge 3$,  $R \ge 4$,  and $\alpha\in (0, 1)$. There are $C = C(n, \alpha)$ and $\eta_0 = \eta_0(n, \alpha)>0$ such that the following holds:

Let $u$ be a classical solution to the Bernoulli problem, \eqref{eq:classical_Bernoulli}, in $D = B_R\setminus \overline{B_1}\subset \R^n$, satisfying
\begin{equation}\label{flatnesshypothesisthm}
    H_0(u,r) \le \eta r\qquad\text{for all}\quad r\in (2, R/2),
\end{equation}
for some $\eta \le \eta_0$.

Then,
\begin{equation}\label{goal}
H(u,r) \le C\eta r\left[ r^{2-n-\alpha} + (r/R)^{\alpha}\right], \qquad\text{for all}\quad r\in (2, R/2).
\end{equation}
\end{theorem}

The proof of the result relies on the following lemma:
\begin{lemma}
\label{lem:Hu1}
    Let  $n\ge3$, $\alpha\in (0,1)$, and $\varepsilon>0$. There exists $\delta = \delta(n, \eps, \alpha)>0$ such that the following holds.  

    Define $\psi(t):= \max(t^{3-n-\alpha}, t^{1+\alpha})$. Let $u$ be a classical solution to the Bernoulli problem in $D = B_{1/\delta}\setminus B_\delta \subset \R^n$, \eqref{eq:classical_Bernoulli}, satisfying,   for some $e\in \mathbb{S}^{n-1}$, $|b|<\delta$, and $0< \eta\le \delta^{2n}$,
    \begin{equation}\label{growth00}
        H_b(u, e, t) \le \eta\, \psi(t) \qquad \text{for all} \quad  2\delta <t<\tfrac 1{2\delta}.
    \end{equation}
      
    Then
    \[
        H(u,1) \le \varepsilon  \eta.
    \]
\end{lemma}
\begin{proof}
Assume without loss of generality $e= e_n$.

If $\delta>0$ is small enough, the classical improvement of flatness for the Bernoulli problem, \cite{desilva2011rhs}, applied at every dyadic scale and rescaled, \eqref{eq:scaling}, yields  
\[
\partial \{u>0\} =  \{x_n = b+  \eta g(x')\} \qquad   \mbox{in}\quad     (B_{\frac{1}{4\delta}}\setminus \overline{B_{4\delta}}),
\]
where we denote $x = (x', x_n)\in \R^{n-1}\times \R$ and
\[
\sup_{B'_{2t}\setminus B'_{t/2}} |g|/t+ |\nabla g| + t|D^2 g| \le C \psi(t)/t
\]
for $8\delta<t< \frac 1 {8\delta}$ and some $C$ dimensional, and where $B'_r$ denotes the ball of radius $r$ in $\R^{n-1}$. We have used here that, since $\eta\le \delta^{n-1+\alpha}$,  $\eta \psi(t)/t<C\delta$ for all $t\in (\delta, 1/\delta)$, $\delta > 0$ is small enough depending only on $n$, and the linear estimates for higher order regularity of the free boundary (see \cite{Kinderlehrer-Nirenberg, Lian-Zhang} or \cite[Lemma 3.7]{CFFS25}).

In particular for any $\tilde e\in \mathbb{S}^{n-1}$ with $|\tilde e-e_n|\le C\eta\psi(t)/t$ and for all $x\in \partial\{u>0\}\cap (B_{2t} \setminus \overline{B_{t/2}})$ we have
\begin{equation}\label{flux1}
\begin{split}
|\partial_{\tilde e}(u-e_n \cdot x)|   & = |\tilde e \cdot \nabla u(x) - \tilde e \cdot e_n| \le  |1-\tilde e \cdot \nabla u(x)| + |1 - \tilde e \cdot e_n|  \\
& = \tfrac12  |\tilde e - \nabla u(x)|^2 + \tfrac12 | \tilde e - e_n|^2   \le C(\eta\psi(t)/t)^2,
\end{split}
\end{equation}
where we have used here that $|1-e_1\cdot e_2| = \tfrac12 |e_1-e_2|^2$ for unit vectors $e_1, e_2\in \mathbb{S}^{n-1}$, the triangle inequality, and the fact that $|\nabla u  - e_n|\le C \eta \psi(t) / t$ on the free boundary in $B_{2t}\setminus \overline{B_{t/2}}$.

Thus, in the coordinates 
\[ y' = x' , \quad y_n = x_n - \eta g(x')-b\]
the function
\[
    v(y):= \frac{u(x(y))- x_n(y)+b}{\eta}
\] 
solves a divergence form equation with Lipschitz coefficients arbitrarily close to the identity in $\{y_n>0\}$, 
\[
{\rm div}_y\left( A(y) \nabla_y v\right) = 0\quad\text{in}\quad B_{\frac{1}{8\delta}}\setminus B_{8\delta}\cap\{y_n > 0\} \qquad \text{where}\quad A(y) :=  \begin{pmatrix}{\rm Id}_{n-1} & -\eta \nabla g(y') \\ -\eta \nabla^\top g(y') & 1+\eta^2|\nabla g(y')|^2\end{pmatrix}.
\]
Moreover using \eqref{flux1} and by assumption \eqref{growth00} we have, for $\delta > 0$ small enough and since $\eta \le \delta^{2n}$, 
\[
\left\{
\begin{array}{rcll}
|v(y)| & \le&  C\psi(|y|) & \qquad\text{in}\quad B_{\frac{1}{8\delta}}\setminus B_{8\delta}\cap \{y_n > 0\},\\
|\partial_{y_n} v(y)| & \le & C \eta  (\psi(|y|)/|y|)^2\le  \delta & \qquad\text{on}\quad B_{\frac{1}{8\delta}}\setminus B_{8\delta}\cap \{y_n = 0\}.
\end{array}
\right.
\]

Now, up to redefining $\delta$ small enough and applying a standard compactness result (see  Lemma~\ref{lem:compactness} below), we have 
\[
    |v(y)  - a \cdot y - d |<\varepsilon/4 \qquad\text{in}\quad (B_4\setminus \overline{B_{1/4}})\cap \{y_n > 0\}, 
\]
for some universally bounded $a\in \R^n$ and $d\in \R$ with $a_n = 0$. In terms of $u$ we have 
\[
|u(x) - x_n + b-\eta \bar a \cdot x-\eta \bar d |\le \eta \eps/2\qquad\text{in}\quad \{u > 0\} \cap (B_2\setminus \overline{B_{1/2}}) 
\]
for some universally bounded $\bar a\in \R^n$ and $\bar d$ with $\bar a \cdot e_n = 0$. That is, $||e_n +\eta \bar a|-1|\le C \eta^2$, and absorbing this term on the right-hand side we obtain the desired result. 
\end{proof}

We have used the following standard compactness lemma:
\begin{lemma}
\label{lem:compactness}
    Let $n \ge 3$ and $\alpha \in (0, 1)$. For any $\eps > 0$ there exists $\delta \in(0,1/4)$ depending only on $n$, $\alpha$, and $\eps$ such that the following holds. 
    
    Define $\psi(t):= \max(t^{3-n-\alpha}, t^{1+\alpha})$. Let $v\in C^{\infty}(\overline{\{x_n > 0\}}\cap (B_{1/\delta}\setminus \overline{B_\delta}))$ satisfy 
    \[
    \left\{
    \begin{array}{rcll}
    {\rm div}(A(x) \nabla v) & = & 0,&\qquad \text{in}\quad \{x_n > 0\}\cap (B_{1/\delta}\setminus \overline{B_\delta}),\\
    |\partial_n v| & \le & \delta & \qquad \text{on}\quad \{x_n = 0\}\cap (B_{1/\delta}\setminus \overline{B_\delta}), 
    \end{array}
    \right.
    \]
       
    \[
    \sup_{x\in (B_{1/\delta}\setminus B_\delta)\cap \{x_n \ge 0\}}\|A(x) - {\rm Id}\|_2 \le \delta, \qquad \text{and}\qquad |v(x)| \le \psi(|x|) \quad \text{for}\quad x\in (B_{1/\delta}\setminus \overline{B_\delta})\cap \{x_n \ge 0\}.
    \]
    Then, there are $a\in \R^n$ and $b\in \R$ such that 
    \[
    |v(x) - a\cdot x - b|\le \eps\qquad\text{in}\quad (B_4\setminus \overline{B_{1/4}})\cap \{x_n \ge 0\},
    \]
    where $a_n = 0$ and $|a|+|b|\le C(n)$.
\end{lemma}
\begin{proof}
    The proof is standard by compactness. Assuming by contradiction that the statement does not hold, there is $\eps_\circ > 0$ and a sequence $v_k$ satisfying the assumptions with $\delta = 1/k$ for $k\to \infty$ such that the conclusion fails with $\eps = \eps_\circ$. 

    By standard regularity estimates and Arzel\`a-Ascoli (see, e.g., \cite{lieberman1983conormal}), $v_k$ converges locally uniformly to $v_\infty$, which satisfies: 
    \[
    \left\{
    \begin{array}{rcll}
    \Delta v_\infty& = & 0,&\qquad \text{in}\quad \{x_n > 0\},\\
    \partial_n v_\infty & = & 0 & \qquad \text{on}\quad \{x_n = 0\}\setminus \{0\}, 
    \end{array}
    \right.
    \]
    where 
    \[
    |v_\infty(x)| \le \psi(|x|) \quad \text{in}\quad\{x_n > 0\}.
    \]
    By taking the even extension across $\{x_n = 0\}$, denoted $\bar v_\infty$, we have that $\Delta \bar v_\infty = 0$ in $\R^n\setminus\{0\}$. The bound by $\psi(|x|)$ implies a removable singularity at the origin, and thus, by Liouville's theorem ($\bar v_\infty$ being subquadratic) we have $\bar v_\infty$ is affine and even in $x_n$. This contradicts the negation of the thesis.
\end{proof}
We can now give the proof of the improvement of flatness in exterior domains for the Bernoulli problem:

\begin{proof}[Proof of Theorem~\ref{thm1}]

We argue by contradiction. Let $C_\star$ be a positive constant, to be chosen large enough (depending only on $n$ and $\alpha$). Put $\phi(t):= t(t^{2-n-\alpha} +(t/R)^{\alpha})$. Define
\[
Q_\star :=  \sup_{r\in [2,R/2]}  \frac{H(u,r)}{\phi(r)},
\]
and assume for contradiction that  
\[
Q_\star \ge C_\star \eta.
\]

By continuity, there exist $r_\star\in[2,R/2]$ and $b_\star\in \R$, $e_\star\in \mathbb{S}^{n-1}$ such that
\[
H(u,r_\star) = Q_\star \phi(r_\star) \qquad\mbox{and}\qquad H(u,r_\star)= H_{b_\star}(u,e_\star,r_\star).
\]
More generally, for every $r\in [2,R/2]$ there exist $b_r\in\R$ and $e_r\in\Sph^{n-1}$  such that
\begin{equation}\label{89grqbgpargh}
    H(u,r) \le  Q_\star \phi(r)=H(u,r_\star)\frac{\phi(r)}{\phi(r_\star)} \qquad\mbox{and}\qquad H(u,r)=H_{b_r}(u,e_r, r)\le\eta r,
\end{equation}
where we have additionally used assumption \eqref{flatnesshypothesisthm} in the last inequality.

We proceed in several steps.

\noindent {\bf Step 1.} Rescaling.

Consider $\bar u := \frac{u(r_\star x)}{r_\star}$. Denoting
$$t = r/r_\star\in (2/r_\star, R/(2r_\star)),\quad \bar b_t = b_r/r_\star,\quad\mbox{and}\quad\bar e_t = e_r,$$ \eqref{89grqbgpargh} becomes (recall \eqref{eq:scaling})
\begin{equation}\label{growthctrl12}
 H_{\bar b_t}(\bar u, \bar e_t, t) \le  \frac{H(u,r_\star)}{r_\star}\frac{\phi(t r_\star)}{\phi(r_\star)} \qquad \mbox{and}\qquad H_{\bar b_t}(\bar u, \bar e_t, t)\leq\eta t.
\end{equation}

Noticing that 
\begin{equation}\label{doubling}
\frac{\phi(t r_\star)}{\phi(r_\star)} \le \max(t^{3-n-\alpha}, t^{1+\alpha}) =: \psi(t),
\end{equation}
this gives
\begin{equation}\label{doubling2}
H_{\bar b_t}(\bar u, \bar e_t, t) \le  \eta_\star \psi(t),\qquad\mbox{with}\quad \eta_\star:= \frac{H(u,r_\star)}{r_\star} \le \eta \le  \eta_0.
\end{equation}
We want to apply Lemma~\ref{lem:Hu1} to $\bar u$. Note that \eqref{flatnesshypothesisthm} implies
\begin{equation}\label{boundbr}
|b_r|< CH(u,r), \quad \mbox{thus} \quad |\bar b_t|< C\frac{H(u, tr_\star)}{r_\star} \le C\eta t.
\end{equation}
\noindent {\bf Step 2.} Coefficient comparison.

By definition of $H_{\bar b_t}(\bar u, \bar e_t, t)$, for  every $t$ as above we have:

\[
\sup_{B_{2t}\setminus  B_{t/2}}  \big|\bar u - (\bar e_t\cdot x +\bar b_t)_+ \big| \le  H_{\bar b_t}(\bar u, \bar e_t, t).
\]

Hence,  using $|\bar b_t| \le t/{100}$  (for $\eta$ small enough) and the triangle inequality we obtain 
\[
\begin{split}
t|\bar e_{2t} -\bar e_{t}| + |\bar b_{2t}-\bar b_{t}| & \le C\sup_{B_{2t}\setminus B_{t}}  \big| (\bar e_{2t}\cdot x +\bar b_{2t})_+ - (\bar e_{t}\cdot x +\bar b_{t})_+ \big|
\\
&\le C\sup_{B_{4t}\setminus B_{t}} \big| \bar u - (\bar e_{2t}\cdot x +\bar b_{2t})_+ \big| + C\sup_{B_{2t}\setminus B_{t/2}}  \big|\bar u - (\bar e_{t}\cdot x +\bar b_{t})_+ \big| 
\\
&\le C (H_{\bar b_t}(\bar u, \bar e_t, t)+ H_{\bar b_{2t}}(\bar u, \bar e_{2t}, 2t)).
\end{split}
\]

By \eqref{doubling2} we then have
\[
\begin{split}
t|\bar e_{2t} -\bar e_{t}| + |\bar b_{2t}-\bar b_{t}|\le C\eta_\star (\psi(2t)+ \psi(t)),
\end{split}
\]
which summing the geometric series yields (recall $e_\star = \bar e_1$), for $2/r_\star <t<R/(2r_\star)$,
\[
t|\bar e_{t}-  e_\star|+ |\bar b_{t}- \bar b_1| \le   C\eta_\star \psi(t).
\]
Combined with \eqref{doubling2}, this implies (recall Definition~\ref{defi:height}) that
\[
H_{\bar b_1}(\bar u, \bar e_1, t) \le \bar \eta_\star\psi(t)\qquad \text{for all}\quad t\in (2/r_\star,R/(2r_\star)), \qquad\mbox{with}\quad \bar \eta_\star:= C\eta_\star\le C \eta_0.
\]
As in the rest of the proof, here $C = C(n, \alpha)$. Moreover, $|\bar b_1| \le C\eta\le C\eta_0$ by \eqref{boundbr}.

\noindent {\bf Step 3.} Estimate for $r_\star$.

By \eqref{flatnesshypothesisthm} we know that $H(u,r)/r\le \eta$ for all $r\in (2,R/2)$, thus
\[
Q_\star \le \frac{\eta}{r_\star^{2-n-\alpha} + (r_\star/R)^{\alpha}}\le \eta\max\left\{\frac{1}{r_\star^{2-n-\alpha}},\frac{1}{(r_\star/R)^{\alpha}}\right\}.
\]
Applying this and $Q_\star\geq C_\star\eta$ twice, we find that
\begin{equation}
    \label{eq:boundonrstar}
C_\star^{\frac{1}{n+\alpha-2}}\le (Q_\star/\eta)^{\frac{1}{n+\alpha-2}}\le r_\star \le R (\eta/Q_\star)^{\frac{1}{\alpha}} \le R C_\star^{-\frac{1}{\alpha}} ,
\end{equation}
so that $r_\star\in [C_\star^{\frac{1}{n+\alpha-2}}, R C_\star^{-\frac{1}{\alpha}}]$. Since our arguments above were valid for all $t\in (2/r_\star,R/(2r_\star))$, given $\delta>0$ this includes all $t\in (\delta,\frac{1}{\delta})$ up to making $C_\star$ large enough.

\noindent {\bf Step 4.} Conclusion.

Let $\eps>0$, to be fixed. Combining all of the above, by choosing $C_\star$ large and $\eta_0$ small (depending only on $n$,  $\alpha$, and $\eps$), the setting of Lemma~\ref{lem:Hu1} is satisfied. This gives
\[
H(\bar u, 1) \le \eps\bar \eta_\star =C\eps\frac{H(u,r_\star)}{r_\star} ,
\]
or (since $\bar u(x)=\frac{u(r_\star x)}{r_\star}$, and thus $H(\bar u, 1)=\frac{H(u,r_\star)}{r_\star}$), equivalently,
$$H(u,r_\star) \le C \eps H(u,r_\star) .$$
Once again, we emphasize that $C=C(n,\alpha)$. Fixing finally $\eps$ so that $C\eps =\frac{1}{2}$, which in turn fixes the choice of $C_\star$ and $\eta_0$ in terms of $n$, and $\alpha$, we arrive at a contradiction. 
\end{proof}

As a consequence, we immediately obtain a first-order expansion at infinity of global solutions with flat free boundaries in exterior domains:

\begin{corollary}\label{cor1}
Let $n\ge 3$ and $\alpha\in (0, 1)$. There are $C = C(n, \alpha)$ and $\eta_0 = \eta_0(n, \alpha)>0$ such that the following holds:

Let $u$ be a classical solution to the Bernoulli problem, \eqref{eq:classical_Bernoulli}, in $\R^n \setminus \overline{B_1}\subset \R^n$, satisfying
\begin{equation}\label{flatnesshypothesisthm0}
    H_0(u,r) \le \eta r\qquad\text{for all}\quad r\ge 2,
\end{equation}
for some $\eta \le \eta_0$.

Then, there is $|b|\le 1/16$ such that, up to a rotation, 
\begin{equation}\label{eq:firstexpansion}
|u - x_n + b|\le \frac{C}{|x|^{n-3+\alpha}}\quad\text{in}\quad \{u > 0\}\setminus B_1.
\end{equation}
Moreover,
    \[
\{u>0\}\setminus (B_2'\times[-2,2]) = \{|x'|> 2,\,   x_n > g(x')\}\cup \{|x'|\le 2,\,   x_n > 2\}
\]
for some smooth $g : \R^{n-1}\setminus B_2' \to [-1,1]$.
\end{corollary}
\begin{proof}
    By Theorem~\ref{thm1} applied with $R  \nearrow \infty$, we know
    \[
    H(u, r) \le C\eta r^\gamma\qquad\text{for all}\quad r\ge 2,
    \]
    where we have denoted $\gamma = 3-n-\alpha < 0$. By definition, there are $e_k\in \mathbb{S}^{n-1}$ and $b_k\in \R$ such that, for any $k \in \N$,
    \[
    (e_k\cdot x - b_k - h_k)_+ \le u \le (e_k \cdot x - b_k + h_k)_+\quad\text{in}\quad B_{2^{k+1}}\setminus B_{2^{k-1}},
    \]
    where $h_k := C \eta 2^{k\gamma}$.  Moreover, by the same reasoning with $k+1$ we deduce that, in the annulus $B_{2^{k+1}}\setminus B_{2^k}\subset B_{2^{k+1}}\setminus B_{2^{k-1}}$ we also have 
    \[
    (e_{k+1}\cdot x - b_{k+1} - h_{k+1})_+ \le u \le (e_{k+1} \cdot x - b_{k+1} + h_{k+1})_+\quad\text{in}\quad B_{2^{k+1}}\setminus B_{2^{k}}.
    \]
    For $\eta$ sufficiently small, this means that the free boundary is trapped inside two different thin strips with directions $e_k$ and $e_{k+1}$, and shifts $b_k$ and $b_{k+1}$. In particular, by the triangle inequality, we must have 
    \[
    2^k |e_{k+1}- e_k| + |b_{k+1}-b_k|\le C (h_k+h_{k+1}) \le C \eta 2^{k\gamma}. 
    \]
    From here, using again the triangle inequality, we deduce 
    \[
    \sum_{i \ge k}|e_{i+1}-e_i| \le C\eta \sum_{i \ge k} 2^{i(\gamma-1)} \le C\eta 2^{k(\gamma-1)},\qquad \sum_{i \ge k}|b_{i+1}-b_i| \le C\eta \sum_{i \ge k} 2^{i\gamma} \le C\eta 2^{k\gamma},
    \]
    where now the constants also depend on $\alpha$. In particular, both sequences $\{e_k\}_{k\in \N}$ and  $\{b_k\}_{k\in \N}$ are Cauchy, and there exist some $e_\infty\in \mathbb{S}^{n-1}$ and $b_\infty \in (-1/16, 1/16)$---making $\eta$ smaller if necessary---such that 
    \[
    |e_\infty - e_k|\le  C \eta 2^{k(\gamma-1)}\quad\text{and}\quad |b_\infty - b_k|\le C \eta 2^{k\gamma}. 
    \]
    Up to a rotation bringing $e_\infty$ to $e_n$, by the triangle inequality this implies 
        \[
    (x_n - b_\infty - h_k)_+ \le u \le (x_n  - b_\infty + h_k)_+\quad\text{in}\quad B_{2^{k+1}}\setminus B_{2^{k-1}},\quad\text{for}\quad h_k = C \eta 2^{k\gamma}.
    \]
    Combining this with the classical improvement of flatness for the Bernoulli problem \cite{desilva2011rhs}, the desired result follows.
\end{proof}

\begin{remark}
    In Corollary~\ref{cor1}, we obtain a shift $b$ because the flatness \eqref{flatnesshypothesisthm0} is compatible with sufficiently small vertical translations. In fact, 
    \[
    |b|\le C \eta, 
    \]
    for some $C$ universal. That is,  $b$ can be made as small as needed by taking $\eta$ small. 
\end{remark}





\subsection{Higher-order expansion for entire solutions}\label{sec:highexp}
The goal of this subsection is to improve Corollary~\ref{cor1} to a higher-order expansion:

\begin{theorem}\label{thm:highordexp}
    Let $n\ge 3$ and $\alpha\in (0, 1)$. In the setting of Corollary~\ref{cor1}, there are $C = C(n, \alpha)$ and $\eta_0 = \eta_0(n, \alpha)>0$ as in the corollary such that, for some $v$ of the form  
\begin{equation}
    \label{eq:voftheform}
    v(x) = x_n- b+\frac{c}{|x|^{n-2}} + \frac{d\cdot x}{|x|^{n}},\quad \text{for}\quad (b, c, d) \in \R\times \R\times \R^n, \quad \text{with}\quad {d}_n=(n-2){ b}{ c},
\end{equation}
we have
    \begin{equation}\label{expansion222}
         \left| u(x) - v(x)\right|\le \frac{C}{|x|^{n-1+\alpha}}\quad \mbox{in}\quad  \{u>0\}\setminus B_2\,.
    \end{equation}
    Moreover, we also have
        \begin{equation}\label{expansion222_grad}
         \left| \nabla u(x) - \nabla v(x)\right|\le \frac{C}{|x|^{n+\alpha}}\quad \mbox{in}\quad  \{u>0\}\setminus B_2\,.
    \end{equation}
\end{theorem}

We now turn to the proof of Theorem~\ref{thm:highordexp}. 
The starting point is Corollary~\ref{cor1}, which gives $b_0$ with $|b_0|\le 1/16$ such that
\begin{equation}\label{initial_aprox}
\big|u-x_n+b_0\big|\le \frac{C}{|x|^{n-3+\alpha}}
\qquad \mbox{in } \{u>0\}\setminus B_1 .
\end{equation}
We next introduce the finite-dimensional family of harmonic profiles used to improve this approximation.

\begin{definition}[Family of Ans\"atze]\label{defi:Ansatz}
Let $n\ge 3$. We define the class of {\em approximating Ans\"atze}, denoted $\cA( \kappa )$, as the class of all functions of the form
\begin{equation}\label{Ansatz1}
{\hat v}(x)=x_n+\hat b + \frac{\hat c}{|x|^{n-2}}+\frac{\hat d' \cdot x'}{|x|^n},
\end{equation}
where  $(\hat b, \hat c, \hat d') \in \R\times \R\times \R^{n-1}$, 
such that 
\begin{equation}\label{Ansatz2}
\max\big( |\hat b|, |\hat c|, |\hat d'|\big)\le \kappa.
\end{equation}
For a given ${\hat v}$ of the form \eqref{Ansatz1}, we denote ${\hat v} = {\hat v}[\hat b, \hat c, \hat d']$. 
\end{definition}

We notice that any ${\hat v}\in \cA(\kappa)$ is harmonic in $\R^n \setminus \{0\}$. Besides being harmonic, the key property that Ans\"atze satisfy is: 
\begin{lemma}
\label{lem:nablav}
Let $n \ge 3$ and $\kappa > 0$. For any ${\hat v}\in \cA(\kappa)$ we have 
\[
    \big| |\nabla {\hat v}|^2 -1 \big| \le C\kappa^2  \quad \mbox{in}\quad  \{ |{\hat v}|<\kappa\}\cap ( B_4\setminus B_{1/4}), 
\]
where $C$ depends only on $n$.
\end{lemma}
\begin{proof}
We first notice that 
\begin{equation}\label{asymp1}
    |x_n|  \le C(|\hat b|+|\hat c| + |\hat d'| + t)  \quad \mbox{in}\quad \{|{\hat v}|<t\}\cap ( B_4\setminus B_{1/4}).
\end{equation}

By an explicit computation 
\[
 \nabla {\hat v} =  e_n + {\hat c} (2-n) \frac{x}{|x|^n}  + \frac{({\hat d}' ,0)}{|x|^n } - n \frac{({\hat d'}\cdot x') x }{|x|^{n+2}}
\]
Hence 
\[
\begin{split}
     \nabla {\hat v}\cdot \nabla {\hat v} &= 1 + 2{\hat c} (2-n) \frac{e_n\cdot x}{|x|^n} -2n \frac{({\hat d'}\cdot x') (e_n\cdot x) }{|x|^{n+2}} + O\left( \kappa^2 \right). \\
\end{split}
\]
But then, using \eqref{asymp1}, in $\{|{\hat v}|< \kappa\}\cap  (B_4\setminus B_{1/4})$ we can replace $e_n\cdot x = x_n$ by $-\hat b$ up to introducing new error terms of size $O(\kappa)$.  By \eqref{Ansatz2}, the lemma follows. 
\end{proof}

The second lemma we need is as follows, where from now on we denote
\begin{equation}
\label{eq:rescalingb0}
u_r (x) :=   \frac{1}{r} u(rx + b_0 e_n),\quad\text{where $b_0$ is given by \eqref{initial_aprox}.}
\end{equation}
  Alternatively, the reader may choose to think that after a translation, $b_0= 0$, and \eqref{eq:rescalingb0} is the usual Bernoulli rescaling. Such a rescaling will be undone in the last part of the last step of the proof of Theorem~\ref{thm:highordexp}, in \eqref{eq:laststepproof}. 
\begin{lemma}\label{lem:theta}
     Let $u$ be as in Theorem~\ref{thm:highordexp}. 
    Let $\theta\in (0,1)$ and ${\hat v}\in \cA(1/8)$ . Suppose that, for some $r\ge 4$ and $0<\hat h\le  1/8$, we have (recall \eqref{eq:rescalingb0}) \begin{equation}\label{setup123}
        |u_r-{\hat v}| \le  {\hat h}  \quad \mbox{in}\quad   \{u_r>0\} \cap \big(B_{4}\setminus B_{1/4}\big)\neq \varnothing.
    \end{equation}
 Then, 
    \[
      |\nabla (u_r-{\hat v})| \le C {\hat h}^\theta  \quad \mbox{in}\quad  \{u_r>0\} \cap \big(B_{2}\setminus B_{1/2}\big),
    \]
    where $C$ depends only on $n$ and $\theta$. 
\end{lemma}
\begin{proof}
 We assume $r = 4$ and write $u$ instead of $u_4$ (as a small abuse of notation), otherwise we argue by scaling. By higher order regularity estimates for the Bernoulli problem (see, e.g., \cite[Proposition 3.7]{CFFS25})  we have that assuming small flatness, $|D^k u|\le C$ in $\{u >0\}\setminus B_{1/4}$, and $\{u > 0\}$ is a $C^k$ domain in $\R^n \setminus B_{1/4}$, for any $k \in \N$.  From the definition of ${\hat v}$, we also have $|D^k {\hat v}|\le C$ in $\R^n\setminus B_{1/4}$. The constant $C$ might depend on $k$. Hence, $u-{\hat v}$ is a small harmonic function in a smooth domain with bounded higher order estimates. By standard interpolation estimates (see, e.g., \cite[Section A.3]{FernandezRealRosOton2024}) we get
    \[
    \begin{split}
    [u-{\hat v}]_{C^{0, 1}(\{u > 0\}\cap B_2\setminus B_{1/2})} & \le C \|u-{\hat v}\|^{1-1/k}_{L^\infty(\{u > 0\}\cap B_3\setminus B_{1/3})}\|u-{\hat v}\|^{1/k}_{C^{k-1,1}(\{u > 0\}\cap B_3\setminus B_{1/3})} \\
    & \le C {\hat h}^{1-1/k} C_k^{1/k}. 
    \end{split}
    \]
    Choosing $k$ large depending on $\theta$ we get the desired result.
\end{proof}

 The following result controls the flux of the difference $u_r - {\hat v}$ along the free boundary. 

\begin{lemma}[Flux control]\label{fluxcontrol}
 Let $u$ be as in Theorem~\ref{thm:highordexp}. 
 Let $\theta\in (0,1)$, $\kappa\in (0, 1/8)$, and ${\hat v}\in \cA(\kappa)$.
Let $r \ge 4$ and assume that \eqref{setup123} holds for some $0<\hat h\le 1/8$.

Then,  recalling \eqref{eq:rescalingb0}, we have 
\[ 
|\nabla (u_r-{\hat v})\cdot \nabla u_r| \le C \big( {\hat h}^{2\theta} + \kappa^2\big) \qquad\text{on}\quad \partial \{u_r>0\}\cap (B_{2}\setminus B_{1/2})\,,
\]
for some $C$ depending only on $n$ and $\theta$.
\end{lemma}
\begin{proof}
 As in Lemma~\ref{lem:theta}, the result  is scaling invariant, so we can prove it for $r=4$. We write $u$ instead of $u_4$.

Notice first that \eqref{setup123} implies
\[
\partial \{u>0 \} \subset \{|{\hat v}|<\hat h\} \quad \mbox{inside }B_2\setminus B_{1/2}.
\]
Since $|\nabla u|^2=1$ on $\partial \{u>0\} \cap B_4\setminus B_{1/4}$  we have
\[
  2(\nabla u \cdot \nabla (u- {\hat v})) =   2  - 2\nabla u\cdot \nabla  {\hat v} = |\nabla u-\nabla  {\hat v}|^2   +  1- |\nabla  {\hat v}|^2
\]
on $\partial \{u>0\}$. Then we apply Lemmas~\ref{lem:nablav} and \ref{lem:theta}.
\end{proof}

\begin{lemma}[Compactness]\label{lem:perturbativecompactness}
Let $n\ge 3$ and $\alpha \in  (0,1)$. For any $\eps > 0$ there exists $\delta>0$ depending only on $n$, $\alpha$, and $\eps$, such that the following holds. 

Let  $g : B'_{1/\delta}\to \R$ with $\|g\|_{C^{3}(B'_{1/\delta})}\le \delta $ and put 
\[
\Omega := \{x_n>g(x')\} \cap  (B_{1/\delta} \setminus B_{\delta}).
\]
Let $w:\Omega\to \R$ be such that
\[
\left\{
\begin{array}{rcll}
\Delta w & = & 0 & \quad\text{in}\quad \Omega,\\
|\partial_\nu w| & \le & \delta&\quad\text{on}\quad \{x_n = g(x')\}\cap \Omega.
\end{array}
\right.
\]
Assume, moreover, that $w: \Omega \to \R$ satisfies 
\[
|w|\le |x|^\alpha + |x|^{1-n-\alpha}\qquad\text{in}\quad \Omega. 
\]

Then, there exist $b_*, c_*\in \R$ and $d_*'\in \R^{n-1}$   with $|b_*|+ |c_*| + |d'_*| \le 50$ such that 
\[
\Big|w - b_* - c_* |x|^{2-n} - (d_*'\cdot  x') |x|^{-n} \Big|\le \eps \quad \mbox{in } (B_4\setminus B_{1/4}) \cap \{x_n >g(x')\}\,.   
\]
\end{lemma}

\begin{proof}
    The proof follows by compactness. Assume by contradiction that the statement does not hold for some sequence $w_k$ and $\delta = 1/k$. 

    By standard $W^{1,p}$ regularity estimates for harmonic functions with $L^p$ Neumann boundary data in smooth domains (see, e.g., \cite{geng2012w1}), $w_k$ converges locally uniformly (by the Sobolev embedding theorem) in $\{x_n \ge 0\} \setminus \{0\}$ to some $w_\infty$ in $\{x_n\ge 0\}$. By weak $W^{1,2}$ convergence, $w_\infty$ is harmonic in $\{x_n > 0\}$ with zero Neumann condition on $\{x_n = 0\}\setminus\{0\}$. If we denote by $\bar w_\infty$ the even extension of $w_\infty$ about $\{x_n = 0\}$ (i.e., $\bar w_\infty(x', x_n) = w_\infty (x', |x_n|)$) we have 
    \[
    \left\{
    \begin{array}{rcll}
        \Delta \bar w_\infty & = & 0& \quad \text{in}\quad \R^n \setminus\{0\}\\
        |\bar w_\infty| & \le  & |x|^\alpha + |x|^{1-n-\alpha} & \quad\text{in}\quad \R^n,\\
        &&\hspace{-1.8cm} \text{$\bar w_\infty$ is even in $x_n$}.
    \end{array}
    \right.
    \]
    In particular, one can take $\bar w_\infty\in L^1_{\rm loc}(\R^n)$ and $\Delta \bar w_\infty$ is a distribution in $\R^n$ supported at $\{0\}$. From the growth condition at zero, it has at most order 1, so we must have 
    \[
    \Delta \bar w_\infty = a  \delta_0 + \sum_{i = 1}^n b_i \partial_i \delta_0
    \]
    in the distributional sense. Together with Liouville's theorem for harmonic functions with sublinear growth (see, e.g., \cite[Proposition 1.19]{FernandezRealRosOton2022}) we deduce 
    \[
    \bar w_\infty = \bar c + \bar a |x|^{2-n} + \sum_{i = 1}^n \bar b_i x_i |x|^{-n}.
    \]
    This is a contradiction for $k$ large enough. The bounds on $\bar a, \bar b$, and $\bar c$, follow from the growth assumption, and $\bar b_n = 0$ by evenness. 
\end{proof}

Before we can give the proof of Theorem \ref{thm:highordexp}, we need to introduce a  fine-grained version of $H(u,r)$ using approximating Ans\"atze as in \eqref{Ansatz1}-\eqref{Ansatz2} (with bounded coefficients).

\begin{definition}\label{defi:new_height} Let $n\ge 3$. Let  $u$ be a classical solution, \eqref{eq:classical_Bernoulli}, in  a domain $\R^n\setminus B_1$    satisfying \eqref{initial_aprox}. Let $r>4$. We define the \emph{fine annular height} at scale $r$  as
\[
     \hat H(u,r) := \min_{{\hat v}\in \cA(1/8)}  \|u_r - {\hat v}\|_{L^\infty( \{u_r>0\}\cap ( B_{2}\setminus \overline{B_{1/2}}))},
\]
where we recall that $\cA(\kappa)$ is given by Definition~\ref{defi:Ansatz}, and $u_r$ is given by \eqref{eq:rescalingb0}.
\end{definition}

\begin{proof}[Proof of Theorem~\ref{thm:highordexp}]
We take as starting point \eqref{initial_aprox} rescaled: for all $r\ge 4$
\begin{equation}\label{startingpt}
    \big|u_r-x_n\big|\le \frac{\overline C}{r^{n-2+\alpha}}
    \qquad\text{in }\{u_r>0\}\cap (B_2\setminus B_{1/2}).
\end{equation}

Set
\[
    \hat h_r:= \hat H(u,r),\qquad r\ge 4.
\]
By \eqref{startingpt} (recall that from Corollary~\ref{cor1} we may assume $|b_0|\le 1/16$), we have
\begin{equation}\label{eq:coarsebound-h}
      r \hat h_r\le \overline C\, r^{3-n-\alpha}\qquad\text{for all }r\ge 4.
\end{equation}

\medskip
We claim that there exists \(C_\star\ge 1\) (large) such that
\begin{equation}\label{eq:fineboundgoal}
   r\hat h_r\le C_\star r^{1-n-\alpha}\qquad\text{for all }r\ge 4.
\end{equation}
To prove this, define (for $C_\star$ to be fixed)
\[
    \eps_\star:=\inf\Bigl\{\eps>0:\  r\hat h_r\le C_\star r^{1-n-\alpha}+\eps
    \quad\text{for all }r\ge 4\Bigr\}.
\]
We argue by contradiction and assume that \(\eps_\star>0\). Then, there exists \(r_\star\ge 4\) such that
\begin{equation}\label{eq:badscale}
    r_\star \hat  h_{r_\star}\ge C_\star r_\star^{1-n-\alpha}+\frac{\eps_\star}{2}.
\end{equation}
Combining \eqref{eq:badscale} with \eqref{eq:coarsebound-h}, we get
$
    C_\star r_\star^{1-n-\alpha}\le \overline C\, r_\star^{3-n-\alpha}
$
and hence
\begin{equation}\label{eq:rstarlarge}
    r_\star\ge \Bigl(\frac{C_\star}{\overline C}\Bigr)^{1/2}.
\end{equation}
That is, by taking $C_\star$ large, we can make $r_\star$ large. 

Let $\delta > 0$ be a small constant to be fixed, that will depend only on $n$ and $\alpha$. By taking \(C_\star\) large enough (depending on $\delta>0$), we may assume
\[
    r_\star\ge 4/\delta.
\]

For any $r >4$, let us  choose coefficients
\[
    (\hat b_r,\hat c_r,\hat d_r')\in [-1/8, 1/8]\times [-1/8, 1/8]\times \overline{ B'_{1/8}}
\]
such that (recall Definition~\ref{defi:Ansatz})
\begin{equation}\label{eq:hr-def}
  \hat h_r= \big\|u_r - {\hat v}[\hat b_r, \hat c_r, \hat d_r']  \big\|_{L^\infty(\{u_r>0\}\cap ( B_{2}\setminus \overline{B_{1/2}}))
  }.
\end{equation}

We now divide the proof into multiple steps. 

\medskip
\noindent{\bf Step 1. 
Control of approximating Ansatz at scale $r_\star$.}
We claim that
\begin{equation}\label{eq:coeff-interior}
|\hat b_{r_\star}|+ |\hat c_{r_\star}|+|\hat d_{r_\star}'|
 \le  C (\overline C)\, r_\star^{-\beta} (\hat h_{r_\star})^{1/2},
 \qquad\text{where}\qquad
 \beta:=\tfrac{n-2-\alpha}{2} >0.
\end{equation}

Indeed, on the one hand by \eqref{eq:hr-def}, for any $r\ge 4$, 
\[
  \bigg|u_{r} - x_n -\hat b_{r} - \frac{\hat c_{r}}{|x|^{n-2}} -\frac{\hat d'_{r}\cdot x'}{|x|^{n}} \bigg|
  \le \hat h_{r}
  \qquad \mbox{in }\Omega_r,
\]
where
\[
\Omega_{r}:= \{u_{r} >0\}\cap (B_2\setminus B_{1/2})
\supset \{ x_n \ge \tfrac 1 8\}\cap (B_{2}\setminus B_{1/2}).
\]

Using that $u_{r/2}(x) =2u_r(x/2)$, and that
\[
\Omega_{r/2}\cap 2 \Omega_{r}\supset \{x_n\ge \tfrac 1 4\}\cap (B_{2}\setminus B_{1}) =: \widetilde\Omega,
\]
the triangle inequality gives 
\begin{equation}
\label{eq:finitelymany}
  \bigg| \hat b_{r/2}+  \frac{\hat c_{r/2}}{|x|^{n-2}} +  \frac{\hat d'_{r/2}\cdot x'}{|x|^{n}} -
  2\hat b_{r}- 2^{n-1}\frac{\hat c_{r}}{|x|^{n-2}} - 2^{n}\frac{\hat d'_{r}\cdot x'}{|x|^{n}}\bigg|
  \le \hat h_{r/2}+ 2\hat h_r
  \qquad \mbox{in }\widetilde\Omega.
\end{equation}

Recall that, by definition of $\eps_{\star}$ and \eqref{eq:coarsebound-h},
\begin{equation}
    \label{eq:geommean}
\hat h_{r}\le C_\star r^{-n-\alpha} + \eps_\star/r
\qquad\text{and}\qquad 
\hat h_{r}\le \overline C r^{2-n-\alpha}.
\end{equation}

In what follows we absorb (harmless) powers of $\overline C$ into $C(\overline C)$.
Setting
\begin{equation}
    \label{eq:hateq}
b_r := r\hat b_r,\qquad c_r := r^{n-1}\hat c_r,\qquad d'_r := r^{n}\hat d'_r,
\end{equation}
and using \eqref{eq:finitelymany} at finitely many points, and \eqref{eq:geommean}, we obtain
\[
\begin{split}
|c_{r/2} - c_r|
&\le C(\overline C)\, r^{n-1}
\Big((C_\star  r^{-n-\alpha})^{1/2}+ (\eps_\star /r)^{1/2}\Big) r^{(2-n-\alpha)/2} \\
&\le C(\overline C) \Big((C_\star)^{1/2} r^{-\alpha}
      + (\eps_\star /r)^{1/2} r^{(n-\alpha)/2}\Big),\\
|d'_{r/2} - d'_r| & 
\le C(\overline C) r^{n}\Big((C_\star  r^{-n-\alpha})+ (\eps_\star /r)\Big)\\
 & \le C(\overline C)\Big(C_\star r^{-\alpha} + (\eps_\star /r)\, r^{n}\Big).
\end{split}
\]

Let now $k_\star\in \mathbb N$ be such that $2^{-k_\star}r_\star \in (8,16)$ and put
\[
r_0 = 2^{-k_\star}r_\star.
\]
Summing the corresponding geometric series (using that $n\ge 3$) we have
\[
\sum_{j=0}^{k_\star-1}(2^{-j}r_\star)^{-\alpha}\le C,
\qquad
\sum_{j=0}^{k_\star-1}(2^{-j}r_\star)^{(n-1-\alpha)/2}\le C\, r_\star^{(n-1-\alpha)/2},\qquad 
\sum_{j=0}^{k_\star-1}(2^{-j}r_\star)^{n-1}\le C\, r_\star^{n-1},
\]
and we thus obtain
\[
\begin{split}
|c_{r_0} - c_{r_\star}| & 
\le C\overline C\Big( (C_\star)^{1/2} + (\eps_\star/r_\star)^{1/2} r_\star^{(n-\alpha)/2} \Big),\\
|d'_{r_0} - d'_{r_\star}| & 
\le C\Big( C_\star + (\eps_\star/r_\star)\, r_\star^{n}\Big).
\end{split}
\]

On the other hand, by \eqref{startingpt},
\[
|c_{r_0}|+ |d'_{r_0}|
\le C  \big(
|\hat c_{r_0}|+ |\hat d'_{r_0}|\big)
\le C (\overline C),
\]
because $r_0\in(8,16)$. Hence, by the triangle inequality (and since $C_\star\ge 1$),
\begin{align}
|c_{r_\star}| & \le  C (\overline C)\Big( (C_\star)^{1/2} + (\eps_\star/r_\star)^{1/2} r_\star^{(n-\alpha)/2}\Big), \notag\\
|d'_{r_\star}| & \le  C (\overline C)  \Big( C_\star + (\eps_\star/r_\star) r_\star^n\Big). \notag
\end{align}

Recalling now \eqref{eq:badscale},
\[
 C_\star\le\hat h_{r_\star} r_\star^{n+\alpha}
 \qquad \mbox{and}\qquad
 \eps_\star/r_\star \le 2\hat h_{r_\star},
\]
we infer (recall also \eqref{eq:hateq})
\[
\begin{split}
|\hat c_{r_\star}|
&= r_\star^{1-n}|c_{r_\star}| \\
&\le C (\overline C)\, r_\star^{1-n}
\Big((C_\star)^{1/2} + (\eps_\star/r_\star)^{1/2} r_\star^{(n-\alpha)/2}\Big) \\
&\le C (\overline C)\Big(
(\hat h_{r_\star})^{1/2} r_\star^{1-n+(n+\alpha)/2}
+ (\hat h_{r_\star})^{1/2} r_\star^{1-n+(n-\alpha)/2}
\Big).
\end{split}
\]
Since
\[
1-n+\frac{n+\alpha}{2} =  -\beta,
\qquad
1-n+\frac{n-\alpha}{2}=\frac{2-n-\alpha}{2}\le -\beta,
\]
it follows that
\begin{equation}
\label{eq:hatcbound}
|\hat c_{r_\star}| \le C (\overline C)\, (\hat h_{r_\star})^{1/2} r_\star^{-\beta}.
\end{equation}

Similarly,
\[
\begin{split}
|\hat d'_{r_\star}|
&= r_\star^{-n}|d'_{r_\star}|  \le C (\overline C)\Big(C_\star r_\star^{-n} + \eps_\star/r_\star\Big)  \le C (\overline C)\Big(\hat h_{r_\star} r_\star^\alpha + \hat h_{r_\star}\Big)
\le C (\overline C)\, \hat h_{r_\star} r_\star^\alpha,
\end{split}
\]
since $r_\star\ge 8$. Finally, using \eqref{eq:coarsebound-h} we obtain
\[
\hat h_{r_\star} r_\star^\alpha
= (\hat h_{r_\star})^{1/2}\Big((\hat h_{r_\star})^{1/2} r_\star^\alpha\Big)
\le C (\hat h_{r_\star})^{1/2} r_\star^{(2-n-\alpha)/2+\alpha}
= C (\hat h_{r_\star})^{1/2} r_\star^{-\beta}.
\]
Therefore,
\[
|\hat d'_{r_\star}| \le C (\overline C)\, (\hat h_{r_\star})^{1/2} r_\star^{-\beta}.
\]

Combining the last two bounds proves the bound for $|\hat c_{r_\star}|$ and $|\hat d'_{r_\star}|$ in \eqref{eq:coeff-interior}. It remains only to prove the bound for $|\hat b_{r_\star}|$. 

We proceed as follows: by \eqref{startingpt} we have, arguing as before,
\begin{equation}\label{decayb}
    \lim_{r\to \infty}   b_r  =0 \quad \mbox{and} \quad |b_r|\le C\overline C r^{3-n-\alpha}.
\end{equation}

On the other hand, \eqref{eq:geommean} and \eqref{eq:badscale} imply: 
\[
r\hat h_r \le 2r_\star \hat h_{r_\star} \quad \mbox{for all}\quad   r \ge r_\star
\]
But then, with a triangle inequality argument similar to that above
\[
|b_{2r}- b_r| \le Cr(\hat h_{2r} + \hat h_{r})\le C r_\star \hat h_{r_\star}  \quad \mbox{ for all $r\ge r_\star$}.
\]

Hence,
\[
|b_{r_\star} - b_R| \le C r_\star \hat h_{r_\star}  \log (R/r_\star).
\]

Now, let us take $R = r_\star^p$ with $p>1$ large. By \eqref{decayb} and triangle inequality we find: 
\[
|b_{r_\star} | \le  |b_{r_\star} - b_R|  + |b_R| \le Cr_\star\big( \hat h_{r_\star}  \log (r_\star) + r_\star^{(3-n-\alpha)p-1} \big)
\]

Let us fix  $p\ge \frac{n+\alpha+1}{n-3+\alpha}$ so that $r_\star^{(3-n-\alpha)p-1}\le C \hat h_{r_\star}$ (recall again \eqref{eq:badscale}). Then, we obtain 
\[
|b_{r_\star} | \le   Cr_\star\, \hat h_{r_\star}  \log (r_\star)  \quad \mbox{and thus}\quad |\hat b_{r_\star}| \le C \hat h_{r_\star}  \log (r_\star) .
\]

But by \eqref{eq:coarsebound-h}   we have $\hat h_{r_\star}\le \overline C r_\star^{2-n-\alpha}$ and thus  
\[
|\hat b_{r_\star} | \le  C\hat h_{r_\star}^{1/2}  r_\star^{-(n-2+\alpha)/2} \log (r_\star) \le C r_\star^{-\beta} \hat h_{r_\star}^{1/2}, 
\]
as we wanted.

\medskip
\noindent{\bf Step 2. Normalization at the bad scale.} 
Let us write
\begin{equation}
    \label{eq:vsupscript}
{\hat v}^r:= {\hat v}[\hat b_r, \hat c_r, \hat d_r']  \qquad \mbox{and}\qquad {\hat v}_\star:= {\hat v}^{r_\star}.
\end{equation}
Since \(r_\star\) is large and \eqref{initial_aprox} holds, the free boundary in the annulus
\[
    B_{r_\star/\delta}\setminus B_{\delta r_\star}
\]
is a \(C^3\) graph in the \(x_n\)-direction with arbitrarily small \(C^3\)-norm after scaling. More precisely, after rescaling by \(r_\star\), we may write
\[\begin{split}
    \Omega_\star
    & := 
    r_\star^{-1}\Bigl(\{u(\, \cdot\,  + b_0 e_n)>0\}\cap (B_{r_\star/\delta}\setminus B_{\delta r_\star})\Bigr)
    \\ & =
     \{u_{r_\star}> 0\}\cap (B_{1/\delta}\setminus B_\delta) \\ &  = \{x_n>g_\star(x')\}\cap (B_{1/\delta}\setminus B_\delta),
    \end{split}
\]
with
\begin{equation}\label{eq:gstar-small}
    \|g_\star\|_{C^3(B'_{1/\delta})}\le \delta,
\end{equation}
provided \(C_\star\) is chosen large enough (again, depending on $\delta$).

Define
\[
    w(x):=\frac{u_{r_\star}-{\hat v}_\star}{ \hat h_{r_\star}}
    \qquad\text{in}\quad \Omega_\star.
\]
Since both \(u_{r_\star}\) and \(\hat v_\star\) are harmonic in \(\{u_{r_\star}>0\}\setminus \{0\}\), we have
\[
    \Delta w=0\qquad\text{in}\quad \Omega_\star.
\]
Also, by \eqref{eq:hr-def},
\begin{equation}\label{eq:w-unit-annulus}
    |w|\le 1\qquad\text{in}\quad \Omega_\star\cap (B_2\setminus B_{1/2}).
\end{equation}

\medskip
\noindent{\bf Step 3. Small Neumann datum.} We apply the flux-control lemma, Lemma~\ref{fluxcontrol},  to \(u_{r_\star}\) and \(v_\star\), 
using \(\theta=\frac34\). After scaling, we obtain, denoting $\rho = r/r_\star$, and recalling \eqref{eq:coeff-interior},
\[
    |\partial_\nu w|
    \le
    \frac{C}{\hat h_{r_\star}}
    \Bigl[\bigl(\hat h_r\bigr)^{2\theta}+C (\overline C)\, r_\star^{-2\beta} \hat h_{r_\star}\Bigr]
     \qquad\text{on}\quad \partial\Omega_\star\cap (B_{2\rho}\setminus B_{\rho/2}).
\]
We remark that the constants above depend on $\delta$ as well. 

Since $\hat h_r \le C_\delta \hat h_{r_\star}$ for all $r\in (\delta r_\star, r_\star/\delta)$ (by definition of $r_\star$ and $\hat h_{r_\star}$; see \eqref{eq:ratio-est}  for a precise estimate in terms of $\rho$) we have, for some $C$ that might depend on $\delta$ (and $\overline C$),
\begin{equation}
\label{eq:firstsecondterm}
    |\partial_\nu w|
    \le
    C\bigl(\hat h_{r_\star}\bigr)^{2\theta-1}
    +Cr_\star^{-2\beta} \qquad\text{on}\quad \partial\Omega_\star\cap (B_{1/\delta}\setminus B_\delta).
\end{equation}
 Since \(\hat h_{r_\star}\le \overline C r_\star^{2-n-\alpha}\) and  we have (for $\theta = \tfrac34$)
\[
\hat h_{r_\star}^{2\theta-1}\to 0\qquad\text{as}\quad r_\star\to\infty,
\] by taking \(C_\star\) large enough we can ensure
\begin{equation}\label{eq:small-neumann}
    |\partial_\nu w|\le \delta
    \qquad\text{on}\quad\partial\Omega_\star\cap (B_{1/\delta}\setminus B_\delta).
\end{equation}

\medskip
\noindent{\bf Step 4. Growth control of \(w\) on all intermediate scales.} We claim that
\begin{equation}\label{eq:w-growth}
    |w(x)|\le C\Bigl(|x|^\alpha+|x|^{1-n-\alpha}\Bigr)
    \qquad\text{in}\quad \Omega_\star,
\end{equation}
for some \(C=C(n,\alpha)\).  Indeed, by the definition of \(\eps_\star\),
\begin{equation}\label{eq:hr-upper}
    r\hat h_r\le C_\star r^{1-n-\alpha}+\eps_\star
    \qquad\text{for all }r\ge 4,
\end{equation}
while \eqref{eq:badscale} also holds. Thus, for \(r=t r_\star\),
\begin{equation}\label{eq:ratio-est}
    \frac{r\hat h_r}{r_\star \hat h_{r_\star}}
    \le
    2\Bigl(t^{1-n-\alpha}+1\Bigr)
    \qquad\text{for}\quad t\in (\delta,1/\delta).
\end{equation}

Let now $s\in (\delta, 1/\delta)$ be fixed, and denote $\rho := s r_\star$. By triangle inequality, we have (recall the notation \eqref{eq:vsupscript})
\begin{equation}\label{eq:alltogether}
\|u_{r_\star} - {\hat v}_\star \|_{L^\infty(\partial B_s \cap \{u_{r_\star} > 0\})} \le s \|u_\rho - {\hat v}^\rho\|_{L^\infty(\partial B_1\cap \{u_\rho > 0\})} + s \|{\hat v}^\rho - {\hat v}_\star(s\,\cdot\,) / s\|_{L^\infty(\partial B_1\cap \{u_\rho > 0\})}.
\end{equation}
On the one hand, we have 
\[
s \|u_\rho - {\hat v}^\rho\|_{L^\infty(\partial B_1\cap \{u_\rho > 0\})}\le s \hat h_\rho \le  2\hat h_{r_\star} (s^{1-n-\alpha}+1),
\]
thanks to \eqref{eq:ratio-est}. On the other hand, recalling \eqref{eq:hateq},
\begin{equation}
    \label{eq:ontheotherhand}
s \|{\hat v}^\rho - {\hat v}_\star(s\,\cdot\,) / s\|_{L^\infty(\partial B_1\cap \{u_\rho > 0\})} \le C_n\left( r_\star^{-1}|b_\rho - b_{r_\star}| + r_\star^{-1}\rho^{2-n}|c_\rho - c_{r_\star}| + r_\star^{-1}\rho^{1-n}|d'_\rho-d'_{r_\star}|\right).
\end{equation}
As in Step 1, whenever \(r,2r\in [\delta r_\star,\delta^{-1}r_\star]\), we have
\begin{equation}\label{eq:coeffcompare}
    |b_{2r}-b_r|+r^{2-n}|c_{2r}-c_r|+r^{1-n}|d_{2r}'-d_r'|
    \le C_{n,\alpha}r(\hat h_{2r}+\hat h_r).
\end{equation}
Summing dyadically, and using \eqref{eq:ratio-est}, we deduce
\[
\begin{split}
|b_\rho - b_{r_\star}|& \le C_{n,\alpha} r_\star \hat h_{r_\star} (s^{1-n-\alpha}+\log s),\\
|c_\rho - c_{r_\star}|& \le C_{n,\alpha} r_\star^{n-1} \hat h_{r_\star} (s^{-1-\alpha}+s^{n-2}),\\
|d'_\rho - d'_{r_\star}|& \le C_{n,\alpha} r_\star^{n} \hat h_{r_\star} (s^{-\alpha}+s^{n-1}).
\end{split}
\]
In \eqref{eq:ontheotherhand} we thus get
\[
s \|{\hat v}^\rho - {\hat v}_\star(s\,\cdot\,) / s\|_{L^\infty(\partial B_1\cap \{u_\rho > 0\})} \le C_{n,\alpha}\hat h_{r_\star} (s^{1-n-\alpha} + |\log s|). 
\]

Since $|\log s|\le C (s^{1-n-\alpha} + s^\alpha)$, putting it all together in \eqref{eq:alltogether} we get \eqref{eq:w-growth}. 

\medskip
\noindent{\bf Step 5. Compactness improvement.} By \eqref{eq:gstar-small}, \eqref{eq:small-neumann}, and \eqref{eq:w-growth}, we may apply the perturbative compactness lemma, Lemma~\ref{lem:perturbativecompactness}. Up to dividing by a constant depending only on $n$ and $\alpha$ for the application of Lemma~\ref{lem:perturbativecompactness}, we may assume that $C$ in \eqref{eq:w-growth} is 1. We obtain coefficients
\[
    b^\sharp,c^\sharp\in\R,
    \qquad
    d^{\sharp\prime}\in \R^{n-1},
\]
bounded by a universal constant (depending on $n$ and $\alpha$, due to the constant in \eqref{eq:w-growth}), such that
\begin{equation}\label{eq:w-compact}
    \Bigl|w(x)-b^\sharp-c^\sharp |x|^{2-n}
    -(d^{\sharp\prime}\!\cdot x')|x|^{-n}\Bigr|
    \le \varepsilon
    \qquad\text{in}\quad\Omega_\star\cap (B_2\setminus B_{1/2}),
\end{equation}
where \(\varepsilon>0\) will be fixed small.

We thus obtain a new Ansatz at scale $r_\star$ given by:
\[
    \widetilde v(x)
    =
    {\hat v}_\star(x)
    +
    \hat h_{r_\star} b^\sharp
    +
    \hat h_{r_\star} c^\sharp |x|^{2-n}
    +
    \hat h_{r_\star} (d^{\sharp\prime}\!\cdot x')|x|^{-n}.
\]
From \eqref{eq:w-compact}, we get 
\[
    |u_{r_\star}-\widetilde v|\le \hat h_{r_\star} \varepsilon  \qquad\text{in }\{u_{r_\star} > 0\}\cap (B_{2}\setminus B_{1/2}).
\]
We can now fix $\eps = \tfrac12$ (which fixes $\delta$ from Lemma~\ref{lem:perturbativecompactness}, depending only on $n$ and $\alpha$) and $r_\star$ large enough so that (also using that $\widetilde v\in \cA(1/8)$ because the coefficients are \emph{small})
\[
    \hat h_{r_\star} = \hat H(u,r_\star)\le \|u_{r_\star}-\widetilde v\|_{L^\infty(\{u_{r_\star} > 0\}\cap (B_2\setminus B_{1/2}))} \le \tfrac12 \hat h_{r_\star};
\]
a contradiction. This proves \(\eps_\star=0\), and therefore
\begin{equation}
    \label{eq:boundhr_fin}
    \hat h_r\le C_\star r^{-n-\alpha}
    \qquad\text{for all }r\ge 4.
\end{equation}

\medskip
\noindent{\bf Step 6. Passage to the asymptotic coefficients.} From \eqref{eq:boundhr_fin} and using \eqref{eq:hr-def}-\eqref{eq:hateq} we know, for all $r\ge 4$ (up to a constant)
\[
\left|u(x+b_0 e_n) - w^r(x)\right| \le r\hat h_r \le C_\star r^{1-n-\alpha} \quad\text{in}\quad \{u(\, \cdot \, + b_0 e_n) > 0\}\cap B_{2r}\setminus B_{r/2}
\]
where 
\[
w^r(x) :=  x_n + b_r +\frac{c_r}{|x|^{n-2}} + \frac{d_r'\cdot x'}{|x|^n}. 
\]
Notice, also, that thanks to Lemma~\ref{lem:theta}, we  know 
\begin{equation}
\label{eq:nabla_control}
|\nabla u(x+b_0 e_n)  - \nabla w^r(x)|\le \hat h_r^\theta\le C r^{\theta(1-n-\alpha)} \quad\text{in}\quad \{u(\,\cdot\,  + b_0 e_n) > 0\}\cap B_{2r}\setminus B_{r/2}.
\end{equation}
Denoting $r_k$ the dyadic scale $r_k := 2^k$, and 
\[
(b_k, c_k, d_k') := (b_{r_k}, c_{r_k}, d'_{r_k}),
\]
by \eqref{decayb}-\eqref{eq:coeffcompare}-\eqref{eq:boundhr_fin} and summing a geometric series, we know $(b_k, c_k, d'_k)\to (0, c_\infty, d'_\infty)$ with 
\[
|b_k|\le Cr_k^{3-n-\alpha},\qquad  |c_k - c_\infty|\le Cr_k^{-1-\alpha},\qquad |d'_k - d'_\infty|\le C r_k^{-\alpha}.
\]
Thus, denoting $w^{r_\infty}$ the corresponding function with coefficients $(0, c_\infty, d'_\infty)$, we have 
\[
\begin{split}
|u(x+b_0 e_n) - w^{r_\infty}(x) |& \le |u(x+b_0 e_n) - w^{r_k}(x) |+|w^{r_k}(x) - w^{r_\infty}(x) |\\ & \le Cr_k^{1-n-\alpha}. 
\end{split}
\]
in $\{u(\,\cdot\, + b_0e_n) > 0\}\setminus B_{r_k/2}$. This yields the following bound, using that $|b_0|$ is bounded,
\begin{equation}
\label{eq:laststepproof}
\left|u(x) - x_n + b_0   - \frac{c_\infty}{|x-b_0e_n|^{n-2}} - \frac{d'_\infty\cdot x'}{|x-b_0e_n|^n}\right|\le \frac{C}{|x|^{n-1+\alpha}}\quad\text{in}\quad \{u > 0\}\setminus B_1. 
\end{equation}
Taking a Taylor expansion we obtain \eqref{eq:voftheform}. Notice, in particular, the appearance of a new term $d_nx_n|x|^{-n}$, with $d_n=(n-2) b_0 c_\infty$. The bound \eqref{expansion222_grad} for the gradients follows in the same way, using \eqref{eq:nabla_control}, up to choosing a larger $\alpha\in (0, 1)$ and $\theta$ very close to 1. 
\end{proof}
In particular, we obtain the: 
\begin{proof}[Proof of Theorem \ref{thm:highordexp_intro}]
It follows from Theorem \ref{thm:highordexp}.
\end{proof}



\section{Asymptotic structure of entire solutions with finite index}\label{sec:asymptind}
We focus now on entire solutions $u:\R^n\to[0,\infty)$ with finite Morse index.
\subsection{Preliminaries}
Our goal is to show:
\begin{proposition}\label{prop:FBdichotomy}
    Let $3\leq n < n_*$, and let $u:\R^n\to[0,\infty)$ be a classical solution to the Bernoulli problem with finite Morse index and which is not one-dimensional.
    
    Then, either $\FB(u)$ is bounded, or there exists some $R_0=R_0(u)>0$ such that we can decompose $u=u_+ +u_-$ in $D:=\R^n\setminus B_{R_0}$, where ${\rm supp}\, u_+\cap {\rm supp}\, u_-=\emptyset$, and each of the $u_\pm$ is a classical solution to the Bernoulli problem in $D$ with
\begin{equation}\label{flatnesshypothesisthm2}
    H_0(u_\pm,r) \le o(r)\quad\mbox{as}\quad r\to\infty.
\end{equation}
\end{proposition}
We first collect some important properties for local and global solutions.

The following gradient bound is well-known (see, e.g., \cite[Proposition A.5]{kamburov2022nondegeneracy}):
\begin{lemma}\label{lem:grad1bound}
    Let $u:\R^n\to[0,\infty)$ be an entire solution to the Bernoulli problem. Then, $|\nabla u|\leq 1$.
\end{lemma}
Let us recall Weiss's monotonicity formula.
\begin{lemma}[{\cite[Theorem 3.1]{Weiss1998}}]\label{lem:weissformula}
    Let $u:B_1\to[0,\infty)$ be an inner stationary solution to the Bernoulli problem, and let 
    \begin{equation}\label{eq:weissdef}
        W( r):=\frac{1}{r^n}\int_{B_r}(|\nabla u|^2+\, \mathbf{1}_{\{u>0\}})-\frac{1}{r^{n+1}}\int_{\partial B_r}u^2\,d\mathcal H^{n-1}
    \end{equation}
    for $r\in(0,1)$. Then $\frac{d}{dr} W(r)\geq 0$; moreover, $\frac{d}{dr} W(r)=0$ precisely if $u$ is one-homogeneous in $B_r$. 
\end{lemma}
We have outer stability for solutions with finite index.
\begin{lemma}[{\cite[Proposition 3.4]{BK24}}]\label{basultokamburov}
    Let $u$ be a classical solution to the Bernoulli problem in $\R^n$ with finite Morse index. Then, there is some $R_0=R_0(u)>0$ such that $u$ is stable in $\R^n\setminus \overline{B_{R_0}}$.
\end{lemma}
We also record a nondegeneracy property.
\begin{lemma}\label{lem:nondegcurv}
    Let $u$ be a classical solution to the Bernoulli problem in $B_1\subset \R^n$ with $0\in \FB(u)$ and $|D^2 u|\leq C_\circ$ in $\{u > 0\}\cap B_1$. Then, there is $C = C(n, C_\circ)$ such that, up to a rotation, $u \ge x_n -Cr^2$ in $B_r$ for all $r \in (0, 1/2)$. 
\end{lemma}
\begin{proof}
    Let us assume $\nabla u(0) = e_n$. It now follows by observing that $v(x) := u(x) - x_n$ satisfies $v(0) = |\nabla v(0)| = 0$ and $|D^2 v|\le C_\circ$ in $\{u > 0\}\cap B_1$. Notice, also, that the Hessian bound implies the same curvature bound on the free boundary by the constant Neumann condition $|\nabla u | = 1$ on $\FB(u)$. 
\end{proof}
We will need the next lemma on solutions that are close to a vee.

\begin{lemma}[{\cite[Lemma 3.15]{CFFS25}}]
\label{cor_closetoV_disc_reg2}
Let $n\ge 2$. Given $C_1\ge 1$ there exists $\eps_1>0$, depending only on $n$ and $C_1$, such that the following holds. 

Let $u$ be a global classical solution to the Bernoulli problem in $\R^n$. Suppose that $|D^2u| \le C_1\varrho^{-1}$ in $B_{2\varrho}\cap \{u>0\}$ and
\begin{equation}\label{sjiowoihw2}
\big|u- V_{0,e_n}\big|\le \eps\varrho \le \eps_1\varrho\quad \mbox{in } B_{2\varrho},
\end{equation}
where $e_n$ is the $n$-th vector in the canonical basis.
Then
\[\varrho^2\|D^2 u\|_{L^\infty(\{u > 0\}\cap B_{\varrho})}\leq C \eps \varrho\]
for some $C$ depending only on $n$. Moreover, 
\[
\{u>0\} = \{x_n > \tilde g^{(+)}(x_1, \dots , x_{n-1})\} \cup \{x_n < \tilde g^{(-)}(x_1, \dots , x_{n-1})\} \qquad\text{in}\quad B_\varrho,
\]
where $\tilde g^{(\pm)} : B_\varrho' \to \mathbb{R}$,  $\tilde g^{(-)}< \tilde g^{(+)}$, and 
\[
\|\tilde g^{(\pm)}\|_{L^\infty(B_\varrho')} + \varrho^2 \|D^2 \tilde g^{(\pm)}\|_{L^\infty(B_\varrho')} \leq C \eps \varrho
\]
for some $C$ depending only on $n$.
\end{lemma}
The following important result shows that entire solutions are close to a wedge solution at all large scales where $\FB(u)$ is present (assuming $u$ is not one-dimensional).

\begin{lemma}[Compactness toward a two-plane configuration]\label{lem:blowdown}
For every $\varepsilon\in(0,1)$ there exists $\delta=\delta(n,\varepsilon)>0$ such that the following holds.

Let $n< n_*$, and let $u$ be a classical solution in $\R^n$, stable in
$\R^n\setminus B_\delta$. Assume that
\[
W(u,2)- W(u,1)\le \delta 
\]
and 
\[
\FB(u)\cap \bigl(B_2\setminus \overline{B}_1\bigr)\neq\varnothing.
\]
Then, there exists a unit vector $e\in \mathbb S^{n-1}$ such that either
\begin{equation}\label{eq:8163985rtgawou}
    \|u-(e\cdot x)_+\|_{L^\infty(B_{16})}\le \varepsilon \qquad \mbox{or} \qquad 
\|u-|e\cdot x|\|_{L^\infty(B_{16})}\le \varepsilon .
\end{equation}
Moreover, in the second case, letting $\mathcal R_e:\R^n\to\R^n$ be a rotation sending $e$ to $e_n$, we have:
\[
\mathcal R_e(\{u>0\}) = \{x_n > \tilde g^{(+)}(x')\} \cup \{x_n < \tilde g^{(-)}(x')\} \qquad\text{in}\quad B_8\setminus \overline{B}_{1/2},
\]
where $\tilde g^{(\pm)} : B_8' \to \mathbb{R}$,   $\tilde g^{(-)}< \tilde g^{(+)}$, and there is $C=C(n)>0$ such that
\[
\|\tilde g^{(\pm)}\|_{L^\infty(B_8')} + C^{-1}\|D^2 \tilde g^{(\pm)}\|_{L^\infty(B_8')} \leq \eps\leq 1/4.
\]

In particular,
\[
\FB(u)\cap \partial B_r\neq\varnothing
\qquad\text{for every }r\in(1,4).
\]
\end{lemma}
\begin{proof} 
    From the assumption that entire stable solutions are flat, it follows (arguing\footnote{The results for curvature estimates are stated for $n = 3$ in \cite{CFFS25}, but the exact same proof works for $n < n_*$ more generally.} as in \cite{CFFS25}) that
    \begin{equation}\label{eq:89g4iubgilb0}
        |D^2 u|\leq \frac{C}{r}\quad \mbox{in}\quad \{0<{\rm dist}(x, \{u = 0\})\leq r/3\}\cap (B_{3r}\setminus \overline{B_{r/3}}),
    \end{equation}
    for every $r\geq 6\delta$.
        
    By Lemmas~\ref{lem:grad1bound} and \ref{lem:nondegcurv}, together with a Harnack chain argument and the fact that $\FB(u) \cap (B_2\setminus \overline{B_1})\neq\varnothing$, we also have:

    \begin{equation}\label{eq:78t287folvwlv}
        u\leq C_M\quad  \mbox{in}\quad B_M,
    \end{equation}
    and
    \begin{equation}\label{eq:15aiosghf09qt40}
        u\geq c_{M}>0\quad  \mbox{in}\quad \left\{{\rm dist}(x, \{u = 0\})\geq M^{-1} \right\}\cap (B_{M}\setminus \overline{B_{M^{-1}}}).
    \end{equation}
    Together with \eqref{eq:89g4iubgilb0}, this shows that
    \begin{equation}\label{eq:89g4iubgilb02}
        |D^2 u|\leq C_M\quad \mbox{in}\quad  (B_{M}\setminus \overline{B_{M^{-1}}}).
    \end{equation}

    Let us show \eqref{eq:8163985rtgawou} by contradiction. Assume that there exist sequences $\delta_i=\frac{1}{i}$, $u_i$, with $W(u_i,2)-W(u_i,1)\leq \delta_i\to 0$, contradicting the thesis. By our estimates above, up to a subsequence $u_i \to u_\infty$ in $C_{loc}^0(\R^n\setminus\{0\})$, where $u_\infty$ is an inner stable and stationary solution. Moreover, $u_\infty \not\equiv 0$, e.g., by passing the nondegeneracy condition \eqref{eq:15aiosghf09qt40} to the limit. In fact, thanks to \eqref{eq:15aiosghf09qt40} we see that the sets $\{u_k>0\}$, $\FB(u_k)$, and $\{u_k=0\}$ converge---locally in the Hausdorff distance sense---to $\{u_\infty>0\}$, $\FB(u_\infty)$, and $\{u_\infty=0\}$.
    
    These properties show that the Weiss energies of the $u_i$ locally converge to that of $u_\infty$; by Weiss's monotonicity formula, $u_\infty$ is 1-homogeneous. Two cases arise:
    \begin{itemize}
        \item If two parallel sheets of $\FB(u_i)$ collapse together (tangentially, by the curvature estimates) at some $y\in \FB(u_\infty)\setminus\{0\}$, then $\frac{1}{r}u_\infty(rx+y)\xrightarrow{r\to 0} |x\cdot e|$ for some $e\in\Sph^{n-1}$. An argument using Weiss's monotonicity formula again, cf. \cite[Proof of Lemma 4.4]{CFFS25}, shows then that $u_\infty$ is also conical around $y$, thus $u_\infty \equiv |(x-y)\cdot e|$, contradicting the negation of \eqref{eq:8163985rtgawou}.
    
        \item Otherwise, $u_\infty$ is a classical solution outside of the origin. By \eqref{eq:nstardef} then $u_\infty$ is a half-space solution, contradicting the negation of \eqref{eq:8163985rtgawou} again.
    \end{itemize}

    Assume now that we are in the second case in \eqref{eq:8163985rtgawou}. Perform a rotation so that $e=e_n$. Up to making $\delta$ smaller, we can then apply Lemma~\ref{cor_closetoV_disc_reg2} around every $x\in (B_8\setminus \overline{B}_{1/2})\cap \FB(u)$, which gives the graphical description of $\{u>0\}$. Finally, from this and $|\tilde g^{\pm}|\leq 1/4$ we find that $\FB(u)\cap \partial B_r\neq\varnothing$ for every $r\in(1,4)$.
\end{proof}
We can now give:
\begin{proof}[Proof of Proposition~\ref{prop:FBdichotomy}]
Assume $\FB(u)$ to be unbounded.
    After a dilation, we can assume $u$ to be stable outside of $B_1$ by Lemma~\ref{basultokamburov}.
    Moreover, from $|\nabla u|\leq 1$ (see Lemma~\ref{lem:grad1bound}) we have $W(R)\leq |B_1|$, and in particular $W(R)$ is bounded; by monotonicity (Lemma~\ref{lem:weissformula}), given $\delta>0$ there exists then $R_0$ such that $W(R)-W(r)\leq \delta$ for every $R_0\leq r\leq R$.
    
    We can then apply Lemma~\ref{lem:blowdown}---appropriately rescaled---to $u$ for every $r\geq R_0$ such that $\FB(u)\cap \partial B_r\neq \emptyset$, with associated $\eps_r\to 0$ as $r\to\infty$. Up to making $R_0$ larger, if $r\geq R_0$ the first case in the Lemma never holds: if, arguing by contradiction, there were $r_k\to\infty$ with 
    $$\|u-(e_{r_k}\cdot x)_+\|_{L^\infty(B_{16r_k})}\leq \eps_{r_k}\xrightarrow{r_k\to\infty}0,$$
    the improvement of flatness result in \cite{desilva2011rhs} would imply that $u\equiv (x\cdot e)_+$ for some $e\in\Sph^{n-1}$, and $u$ would be one-dimensional.

    Let now $r\geq R_0$ with $\FB(u)\cap \partial B_r\neq \emptyset$, which we can find since we are assuming $\FB(u)$ to be unbounded. By the previous argument, the second case in Lemma~\ref{lem:blowdown} applies, i.e.
    $$\|u-|e_{r}\cdot x|\|_{L^\infty(B_{16r})}\leq\eps_r\xrightarrow{r\to \infty} 0.$$
    Moreover, $\FB(u)\cap \partial B_{s}\neq \emptyset$ for every $s\in (r,4r)$. But then, in particular we can apply Lemma~\ref{lem:blowdown} again with $2r$ in place of $r$; iterating this, we conclude that in fact $\FB(u)\cap \partial B_r\neq \emptyset$ for every $r\geq R_0$. Moreover, putting $A_r:=B_{8r}\setminus B_r$, Lemma~\ref{lem:blowdown} gives a decomposition $\{u>0\}\cap A_r=U_r^+\cup U_r^-$, with associated $e_r\in\Sph^{n-1}$ such that
    \[
\mathcal R_{e_r}(U_r^{+}) = \{ x_n > \tilde g_r^{(+)}(x')\}\cap A_r,\qquad \mathcal R_{e_r}(U_r^{-}) = \{ x_n < \tilde g_r^{(-)}(x')\}\cap A_r,
\]
with
$\tilde g_r^{(\pm)} : B_{8r}' \to \mathbb{R}$,  $\tilde g_r^{(-)}< \tilde g_r^{(+)}$, and
\[
\|\tilde g_r^{(\pm)}\|_{L^\infty(B'_{8r})} + C^{-1}\|D^2 \tilde g_r^{(\pm)}\|_{L^\infty(B'_{8r})} \leq r\eps_r\leq r/4.
\]
By considering the overlap among the $A_r$, we see that---up to changing the sign of some of the $e_r$, and making $\eps_r$ small enough by taking $r$ large enough---we can ensure that $|e_{\lambda r}-e_r|<1/8$ for all $\lambda\in(1/128,128)$, thus also $U_{\lambda r}^+\cap U_r^-=\emptyset$. Defining $U^+=\cup_{r\geq R_0} U^+_r$ and $U^-=\cup_{r\geq R_0} U^-_r$ it follows that $U^+\cap U^-=\emptyset$, and we put $u_+=u\, \mathbf{1}_{U^+}$ and $u_-=u\, \mathbf{1}_{U^-}$. Finally, letting $e_r^\pm = \pm e_r$, since $|u-|e_r\cdot x||\leq o(r)$ in $B_{2r}\setminus B_{r/2}$ it follows that $|u_\pm-(e_r^\pm\cdot x)_+|\leq o(r)$ as well---for instance, by a compactness-plus-contradiction argument. This shows that $H_0(u_{\pm},r)\leq o(r)$ and concludes the proof.
\end{proof}

\subsection{Asymptotic structure in the non-compact free boundary case}\label{sec:asympstruc}
Combining our results up to now, we obtain a detailed expansion for $u$ in the non-compact case:

\begin{corollary}\label{cor:errplus0}
    Let $3\leq n < n_*$ and $\alpha\in(0,1)$. Let $u:\R^n\to[0,\infty)$ be a classical solution to the Bernoulli problem with finite Morse index, non-compact free boundary, and which is not one-dimensional.

    After a rescaling, rotation, and translation, we have the decomposition $u=u_++u_-$ in $D:=\R^n\setminus (B_{1}'\times [-1,1])$, where each of the $u_\pm$ is a classical solution to the Bernoulli problem in $D$. Letting $\Omega_\pm=\{u_\pm>0\}\subset D$, we have $\Omega_+\cap \Omega_-=\emptyset$ and
    \[
\Omega_\pm = \{|x'|> 1,\,   \pm x_n > g^{(\pm)}(x')\} \cup \{|x'|\le 1,\,   \pm x_n > 1\} ,
\]
for some smooth $g^{(\pm)} : \R^{n-1}\setminus B_1' \to [-1,1]$.

Moreover, the $u_\pm$ satisfy
   \[
  |u_\pm -  v_\pm |  + |x| | \nabla (u_\pm-v_\pm)|\le  |x|^{1-n-\alpha} \quad \mbox{in }\Omega_\pm,
  \]
  where 
  \[v_\pm =  \pm x_n-  b  +\frac{c_\pm}{|x|^{n-2}} + \frac{d_\pm\cdot x}{|x|^{n}},\]
 with $b\geq 0$, $c_\pm\in\R$, $d_\pm\in \R^n$,  and satisfying $(d_\pm)_n = \pm (n-2)bc_\pm$  and $|b| +|c_\pm| + |d_\pm | \le 1/2$.
\end{corollary}

\begin{proof}
    Decompose $u=u_++u_-$ via Proposition~\ref{prop:FBdichotomy}. By \eqref{flatnesshypothesisthm2}, we can apply Corollary~\ref{cor1} and Theorem~\ref{thm:highordexp}---appropriately rescaled---to each of $u_+$ and $u_-$, separately. We obtain Ans\"atze $v_\pm=e_\pm\cdot x-  b_\pm  +\frac{c_\pm}{|x|^{n-2}} + \frac{d_\pm\cdot x}{|x|^{n}}$, where $e_\pm\in\Sph^{n-1}$, such that
    $$|u_\pm- v_\pm|+|x||\nabla (u_\pm- v_\pm)|\leq C|x|^{1-n-\alpha}\quad \mbox{in}\quad \Omega_\pm,$$
    and an associated graphicality condition for the $\Omega_\pm$ in the directions $e_\pm$. Now, since ${\rm supp}\, u_+ \cap {\rm supp}\, u_- = \emptyset$, it follows that necessarily $e_-=-e_+$; hence, after a rotation we have $e_\pm=\pm e_n$. From this, we see that $(d_\pm)_n = \pm (n-2)bc_\pm$ (by the relation between the coefficients from Theorem~\ref{thm:highordexp}) and that $\Omega_\pm$ satisfy the graphicality condition in $x_n$ from the statement. Similarly, necessarily $b_+\geq -b_-$, and considering the translated solution $u(x', x_n+\frac{b_+-b_-}{2})$ instead we can impose that $b_+=b_-=:b\geq 0$ as well. Notice that by translating we are also altering the coefficients, e.g., we need to expand $c_\pm |x+\tfrac{b_+-b_-}{2}e_n|^{2-n}$ around $c_\pm |x|^{2-n}$, but the relationship between them remains (one sees this either by direct computation, or by just reapplying Theorem~\ref{thm:highordexp}). Up to rescaling again so that the bounds on the $b,c_\pm,d_\pm$ and the error terms hold, we obtain the result as stated.  
\end{proof}

We find then a natural lower barrier for our solution:
\begin{lemma}\label{lem:barrier}
In the setting of Corollary~\ref{cor:errplus0}, it holds $u\ge |x_n| -b$.
\end{lemma}
\begin{proof}
It is enough to show that $u \ge x_n-b$ in $\R^n$; indeed, arguing symmetrically, $u \ge -x_n-b$ as well. Let
\[
h_\ast := \inf\{h > 0: u \ge x_n - h \quad\text{in}\quad \R^n\}.
\]
By Corollary~\ref{cor:errplus0} we have
\begin{equation}
    \label{eq:expansion_u_zeroth}
    u \geq x_n - b- |x|^{2-n}\quad\text{in}\quad \R^n\setminus B_1, 
\end{equation}
thus $h_\ast<\infty$. Moreover, it is a minimum by continuity, thus it suffices to show that $h_\ast\leq b$. Suppose that $h_\ast > b$ instead; then
\[
 u \ge x_n - b- |x|^{2-n} > x_n - h_\ast+\eps\quad\text{in}\quad \R^n\setminus B_R
\]
for $\eps = (h_\ast-b)/2$ and some $R > h_\ast$ large enough (so that $h_\ast > b + 2R^{2-n}$). That is, we have two ordered harmonic functions $u\ge v:= x_n - h_\ast$ in $\{x_n > h_\ast\}\cap B_R$, with $u > v+\eps$ on $\partial B_R$. By definition of $h_\ast$ (and the fact that $u$ is 1-Lipschitz in $\R^n$) and harmonicity, we must then have $u = v = 0$ at some point on $\{x_n = h_\ast\}\cap B_R$. But this is a contradiction with Hopf's lemma, by the Bernoulli condition on $u$ and the fact that $\partial_n v = 1$.
\end{proof}

The coefficients further satisfy:
\begin{lemma}\label{lem:precisectt}
     In the setting of Corollary~\ref{cor:errplus0}, we have $b, c_\pm,  \pm (d_\pm)_n>0$.

     In particular, up to a rescaling we can additionally assume that $  g^{(\pm)}>0$ in Corollary~\ref{cor:errplus0}, thus $\Omega_\pm\subset\{\pm x_n>0\}$.
\end{lemma}

\begin{proof}
The fact that $b > 0$ follows from Lemma~\ref{lem:barrier}: otherwise $u \ge |x_n|$, which would contradict the assumption that $u$ be a classical solution with unbounded free boundary. Together with the asymptotics for $g^{(\pm)}$ in Corollary~\ref{cor:errplus0}, it follows that $  g^{(\pm)}(x')>0$ in $\R^{n-1}\setminus B_M$ for some $M>0$, thus in $\R^{n-1}\setminus B_1$ up to a rescaling.

It suffices then to show that $c_+,(d_+)_n>0$, since the corresponding properties for $c_-,(d_-)_n$ follow then by symmetry.

First, from Corollary~\ref{cor:errplus0} and $u \ge x_n - b$ we have 
\[
-c_+ |x|^{2-n}\le (u - x_n + b) -c_+ |x|^{2-n}\le C |x|^{1-n}\quad\text{in}\quad \{x_n> b\}.
\]
Letting $|x|\to\infty$, this shows that $-c_+\leq 0$, i.e. $c_+ \geq 0$.

If $c_+>0$, since $(n-2) b c_+  = (d_+)_n$ we see that $(d_+)_n>0$ as well, concluding the proof. Otherwise, we would have $c_+=(d_+)_n=0$; let us show that this leads to a contradiction. First, by Corollary~\ref{cor:errplus0} we find 
\[
 0< w \le |x|^{1-n}\left( C|x|^{-\alpha}+\frac{d'_+\cdot x'}{|x|}\right)\quad\text{in}\quad \{x_n  > b\},
\]
where $w := u - x_n + b  > 0$ is harmonic in $\{x_n > b\}$. Considering $x=(-\lambda d_+',2b)$ and sending $\lambda\to\infty$ shows that $d_+' = 0$, thus
\[
0 \leq w \le C|x|^{1-n-\alpha}\quad\text{in}\quad \{x_n > b\}.
\]
If $w>0$, we obtain a contradiction with the maximum decay of order $|x|^{1-n}$---given by the Poisson kernel of the half-space---for positive harmonic functions. Otherwise $w(x)=0$ at some point, thus $u\equiv x_n-b$ in $\{x_n\geq b\}$  by the (strong) maximum principle. 
Then, $\bar u = u \, \mathbf{1}_{\{x_n > -b\}}$ is a classical entire solution to the Bernoulli problem. By improvement of flatness, the component of $\{\bar u > 0\}$ containing $\{u_+ > 0\}$ must be flat, and there cannot be  other components, as they would be bounded. 
\end{proof}

\section{Symmetry for entire solutions with finite index}\label{sec:symindex}

We now use the asymptotic information obtained above to prove the symmetry of finite-index entire solutions. 
The proof separates the compact and non-compact free boundary cases.

\subsection{Compact free boundary case} The compact case reduces to the classical exterior overdetermined problem:
\begin{proposition}\label{prop:compactclass}
    Let $u$ be a classical solution to the Bernoulli problem in $\R^n$, $n\geq 3$, with compact free boundary. Then, either $u$ is constant, or after a translation and rescaling we have $u=\frac{1}{n-2}(1-|x|^{2-n})_+$.
\end{proposition}
\begin{proof}
    We can assume that $\FB(u)$ is nonempty: otherwise $u$ is a nonnegative harmonic function in $\R^n$ and thus constant. Since $\FB(u)$ is compact, then $\{u>0\}$ consists of a (single) unbounded component. Let $R_0$ be such that $\FB(u)\subset B_{R_0}$, so that $u$ is a positive harmonic function in $\R^n\setminus B_{R_0}$, satisfying $|u|\leq C(1+|x|)$ since $|\nabla u|\leq 1$. Separation of variables\footnote{Alternatively one can use a Green's function representation, or a (quite simpler) version of the arguments in Section~\ref{sec:improvflat}.} gives then $a,d\in\R^n$ and $b,c\in\R$ such that
    $$u=a\cdot x+b+ \frac{c}{|x|^{n-2}}  + \frac{d\cdot x}{|x|^{n}} + O(|x|^{-n})\qquad\mbox{in}\quad \R^n\setminus B_{R_0}\,.$$
    Since $u\geq 0$, necessarily $a=0$. Combined with standard harmonic estimates, $v:=b-u$ satisfies $|v|+|\nabla v|\to 0$ as $|x|\to\infty$, and $v\geq 0$ by the maximum principle.
    
    Then, the main result in \cite{reichel1997radial}---whose proof is based on the moving planes method---shows that $v$ (hence, $u$) is radially symmetric, and $\{v=b\}$ (hence, $\{u = 0\}$) is a ball.
    
    Let $R$ be the radius of this ball; after a rescaling and translation, we can assume that $R=1$ and it is centered at the origin. Integrating radially the equation $\Delta u(r)=0$---which is now an ODE, since $u=u(r)$ is radial---for $r\geq 1$ and imposing the initial conditions $u(1)=0$, $u'(1)=1$ yields the desired form for $u$.
\end{proof} 

\subsection{Non-compact free boundary case: symmetry in $x_n$}\label{sec:symxn}
Let us now assume that $3\leq n < n_*$, and that $u:\R^n\to[0,\infty)$ is a classical solution to the Bernoulli problem with finite Morse index, non-compact free boundary, and which is not one-dimensional.

We can (and do) assume then that we are in the setting given by Proposition~\ref{prop:FBdichotomy} and Corollary~\ref{cor:errplus0} (recall also Lemmas~\ref{lem:barrier} and \ref{lem:precisectt}). In particular
\begin{equation}\label{eq:56290tyq}
    \Omega_\pm:=\{u_\pm>0\}\subset D:=\R^n\setminus (B_1'\times[-1,1]),
\end{equation}
and the $\Omega_\pm$ are given by 
\begin{equation}\label{eq:758915gsA}
\Omega_+ = \{|x'|> 1,\,     x_n > g^{(+)}(x')\} \cup \{|x'|\le 1,\,     x_n > 1\} ,
\end{equation}
\begin{equation}\label{eq:758915gsB}
\Omega_- = \{|x'|> 1,\,    x_n < -g^{(-)}(x')\} \cup \{|x'|\le 1,\,    x_n <- 1\} ,
\end{equation}
with (by Lemma~\ref{lem:precisectt})
\begin{equation}\label{eq:758915gsC}
    g^{(+)}>0 \quad \mbox{and} \quad g^{(-)}>0 .
\end{equation}
 
\begin{remark}
    In fact, the simpler asymptotics in Corollary~\ref{cor1} would suffice in this section.
\end{remark}

This section will show:
\begin{proposition}\label{prop:symxn}
    In the setting of Corollary~\ref{cor:errplus0}, $u$ is even in $x_n$. Moreover, $\partial_{x_n} u> 0$ in $\{x_n> 0\}\cap \overline{\{u>0\}}$.
\end{proposition}

We start with some asymptotic properties.
\begin{lemma}[Gradient asymptotics]\label{lem:gradn}
There is $M>0$ such that
$$\partial_{x_n} u_+>3/4\quad\mbox{in}\quad \Omega_+\cap \{|x|\geq M \}\qquad\mbox{and}\qquad \partial_{x_n} u_-<-3/4\quad\mbox{in}\quad \Omega_-\cap \{|x|\geq M \}.$$
\end{lemma}
\begin{proof}
The result follows from \eqref{eq:758915gsA}--\eqref{eq:758915gsC} and the gradient comparison in Corollary~\ref{cor:errplus0}.
\end{proof}

In the following, we define for $x= (x', x_n)$ the symmetry:
\[
\sigma_t (x) := (x' , 2t-x_n)\quad\mbox{and}\quad  x^*_t := \sigma_t (x),
\]
\[
u^*_t = u \circ \sigma_t \qquad \mbox{i.e.,}\quad  u_t^*(x', x_n) := u(x', 2t-x_n). 
\]

\begin{lemma}[Vertical initialization and reflection asymptotics] \label{lem:refln}
In the setting of Corollary~\ref{cor:errplus0}, we have:

\noindent (a) There exists $t_0>8$ (large) such that for all $t>t_0$ we have:
$$
u\geq u_t^*\qquad\mbox{in}\quad\{x_n\geq t\}
$$

\noindent (b) Given $\delta >0$, there is $R_\delta \ge 0$  such that if $t\geq \delta$, 
$$
u\geq u_t^*\qquad\mbox{in}\quad\{x_n\geq t\}\setminus B_{R_\delta}.
$$
\end{lemma}

\begin{proof}
We first show (a).  Let $x\in \{x_n \ge t\}$ with $t\ge t_0>8$; by Corollary~\ref{cor:errplus0} we have $u(x)>0$. If $u_t^*(x) = 0$ there is nothing to prove, so we can also assume $u_t^*(x) > 0$, that is, $\sigma_t(x) \in \{u > 0\}$. We now have three options, where we denote by $M \ge 1$ the constant from Lemma~\ref{lem:gradn}:  

\begin{itemize}
\item If $\sigma_t(x) \in \{u > 0\}\cap {B_{M}} $, since $u$ is 1-Lipschitz (and    $\partial B_2\cap \{u=0\}\neq \varnothing$) we know $u_t^*(x) \le 4 M$. On the other hand, $u(x) \ge |x_n| -b \ge  t_0-b$, so that taking $t_0$ large we get $u(x) \ge u_t^*(x)$.  

\item If $\sigma_t(x) \in  \Omega_+\setminus {B_{M}}$, then $0<(\sigma_t(x))_n<x_n$, $|\sigma_t(x)|\geq M$, and (by \eqref{eq:758915gsA}--\eqref{eq:758915gsC})   $u>0$ on the segment between $\sigma_t(x)$ and $x$. Lemma~\ref{lem:gradn} gives that $u(x) \ge  u_t^*(x)$ as well.

\item Finally, if $\sigma_t(x) \in \Omega_-\setminus {B_{M}}$, since $x\in \Omega_+\setminus B_{M}$ Corollary~\ref{cor:errplus0} gives (up to maybe making $M$ universally larger) 
\[ u(x) = u_+(x) \ge x_n -b -1\qquad\text{and}\qquad 
u^*_t(x) =  u_-(\sigma_t(x)) \le  x_n-2t -b  +1.
\]
From $t_0>1$ (so that $-2t+1\leq -2t_0+1\leq -1$) we conclude that $u(x)\geq u_t^*(x)$ in this case as well.
\end{itemize}

\medskip
Let us now show (b). If $t>t_0$, where $t_0>8$ was found in (a), we already conclude. Hence, it suffices to show that given $\delta >0$, there exists $R_\delta \ge t_0$ such that, for $t\in (\delta, R_\delta)$, we have
$$
u\geq u_t^*\qquad\mbox{in}\quad\{x_n\geq t\}\setminus (B_{R_\delta}'\times(0, 3R_\delta)).
$$
Let $x\in \{x_n \ge t\}\setminus  (B_{R_\delta}'\times(0, 3R_\delta))$. If $u_t^*(x) = 0$ there is nothing to prove, so we can assume that $u_t^*(x) > 0$. From $R_\delta>1$ one sees that $\sigma_t(x) \in \R^n\setminus (B_1'\times [-1,1])$, thus $\sigma_t(x) \in \Omega_+\cup \Omega_-$. 

Let us assume first that $u(x)>0$ as well; we will reduce to this case immediately after. By \eqref{eq:758915gsA}--\eqref{eq:758915gsC} we have $x\in \Omega_+$.

If also $\sigma_t(x) \in \Omega_+$, Lemma \ref{lem:gradn} shows that $u(x) > u(\sigma_t(x)) = u_t^*(x)$, just as in case (a). If $\sigma_t(x) \in \Omega_-$ instead, by the asymptotics in Corollary~\ref{cor:errplus0} we can ensure, up to making $R_\delta$ large enough, that
\[ u(x) = u_+(x) \ge x_n -b -\delta/3\qquad\text{and}\qquad 
u^*_t(x) =  u_-(\sigma_t(x)) \le  x_n-2t -b  +\delta/3.
\]
Since $t\geq \delta$ this shows that $u(x)>u^*_t(x)$ also in this case.

It only remains to argue that $u(x)>0$ indeed above. If $\sigma_t(x)\in\Omega_+$, this is automatic by Lemma~\ref{lem:gradn}. Assume that $\sigma_t(x)\in\Omega_-$ and $u(x)=0$ instead; increasing $x_n$ until we reach $\FB(u)$, and using the case from before, we then find a contradiction. More precisely: let $s:=\max\{\lambda\in[x_n,\infty):u(x',\lambda)=0\}$ and $y=(x',s)\in\R^n$, so that $y\in {\rm graph}\, g^{(+)}\subset \FB(u)$ (recall \eqref{eq:758915gsA}--\eqref{eq:758915gsC}). Then $u(y+he_n)>0\ \forall\ h>0$; moreover, by Lemma~\ref{lem:gradn} applied with $u_-$ we have $u(\sigma_t(y+he_n))\geq u(\sigma_t(x))>0$ as well. We can therefore apply the previous case with $y+he_n$ in place of $x$, finding
$$u(y+he_n)\geq u(\sigma_t(y+he_n))\geq u(\sigma_t(x))>0\quad \forall \ h>0.$$
Sending $h\to 0$ then $u(y)\geq u(\sigma_t(x))>0$, a contradiction with $y\in\FB(u)$. 
\end{proof}

In the proof of Proposition~\ref{prop:symxn}, we will need the following result from \cite{Serrin71}, which we reproduce here for the reader's convenience:
\begin{lemma}[{\cite[Lemma 1]{Serrin71}}]
\label{lem:serrin}
Let $U^{*}$ be a domain with $C^{2}$ boundary and let $H$ be a plane
containing the normal to $\partial U^{*}$ at some point $x\in \partial U^*$.
Let $U$ then denote the portion of $U^{*}$ lying on some particular side of $H$.

Suppose that $v$ is of class $C^{2}$ in the closure of $U$ and satisfies
\[
\Delta v \leq 0 \qquad \text{in}\quad  U,
\]
while also $v \ge 0$ in $U$ and $v = 0$ at $x$.
Let $s$ be any direction at $x$ which enters $U$
non-tangentially. Then either
\[
\frac{\partial v}{\partial s}(x) > 0
\qquad \text{or} \qquad
\frac{\partial^{2} v}{\partial s^{2}}(x) > 0,
\]
unless $v \equiv 0$.
\end{lemma}

We can now proceed with:
\begin{proof}[Proof of Proposition~\ref{prop:symxn}]
Let $T\subset [0, \infty)$ be defined as 
\[
T := \{t \in [0, \infty) : u \ge u_t^*\ \  \text{in} \ \  \{x_n\ge t\}\quad \text{and}\quad \partial_{x_n} u > 0\ \  \text{in} \ \  \{x_n> t\}\cap\overline{\{u > 0\}}\, \}. 
\]
 We want to show that $0\in T$: This gives  $u\geq u_0^*$ and $\partial_{x_n} u > 0$ in $\{x_n > 0\}\cap \overline{\{u > 0\}}$, and arguing symmetrically we deduce  $u=u_0^*$, i.e., $u$ is even in $x_n$, as we wanted. 

We start by observing that, if $t\in T$ and $s>t$, then $s\in T$ as well. To show it, it is enough to verify that $u\geq u_s^*$ in $\{x_n\geq s\}$---we will repeatedly use that $u$ is monotone nondecreasing in the $x_n$ direction in $\{x_n > t\}$ (since there it is monotone when positive, and it is nonnegative):
\begin{itemize}
    \item If $s\leq x_n \leq 2s-t$, so that $t\leq 2s-x_n\leq s$, we have 
    $$u_s^*(x', x_n)=u(x',2s-x_n)\leq u(x',s)\leq u(x',x_n).$$
    \item If $x_n\geq 2s-t$, so that $x_n-2(s-t)\geq t$, we have
    $$u_s^*(x', x_n)=u(x',2s-x_n)=u_t^*(x',x_n-2(s-t))\leq u(x',x_n-2(s-t))\leq u(x',x_n),$$
    by monotonicity and the fact that $u\geq u_t^*$ in $\{x_n\geq t\}$.
\end{itemize}
Hence, $T$ is an interval (a half-line). Notice also that, by definition and continuity of $u$, $T$ is closed. Moreover, $T$ is nonempty by (a) in Lemma~\ref{lem:refln} and Lemma~\ref{lem:gradn}. We will show that $T$ is relatively open in $[0, \infty)$, thus concluding that $0\in T$.

Let $t>0$, with $t\in T$. By definition, $\partial_{x_n} u\geq 0$ in $\{x_n= t\}\cap \overline{\{u>0\}}$: indeed, the only nontrivial case to consider is if $x\in \FB(u)\cap \{x_n= t\}$. The fact that $u\geq u_t^*$ in $\{x_n \ge t\}$ applied around $x$ then forces $\partial_{x_n} u (x) = \nu_{\FB(u)}(x)\cdot e_n\geq 0$.

\medskip

\noindent {\bf Step 1.} We show that $\partial_{x_n} u>0$ also on $\{x_n= t\}\cap \overline{\{u>0\}}$.

Assume there exists some $x\in \{x_n= t\}\cap \overline{\{u>0\}}$ with $\partial_{x_n} u (x)=0$ instead. Consider $v:=u-u_t^*$, so that $v\geq 0$ in $\{x_n\geq t\}$, $v=0$ on $\{x_n=t\}$, $\partial_{x_n} v(x)=0$, and $v>0$ in $ \{u_t^*>0\}\cap \{x_n > t\}$ (i.e., where $v$ is harmonic, since $u\geq u_t^*$ in $\{x_n > t\}$). 
 There are two cases:
\begin{itemize}
\item If $u(x)>0$, Hopf's lemma (applied to $v$ in a small half-ball $D:=B_r(x)\cap \{x_n\geq t\}\subset \{u>0\}\cap \{u_t^*>0\}$)  gives  $\partial_{x_n} v(x)>0$, which contradicts the fact that $\partial_{x_n} v(x)=0$.

    \item If $u(x)=0$ instead, we have $x\in \FB(u) \cap \{x_n = t\}$. Observe that $\nabla u(x)=\nabla u_t^*(x)$ (i.e., $\nabla v(x)=0$) since $\partial_{x_n}u(x) = 0$.  
    This shows that $\FB(u)$ and $\FB(u_t^*)$ are tangent at $x$.   Put $U^*=\{u_t^*>0\}$ and $H=\{x_n = t\}$, so that $x\in \partial U^*\cap H$ and $\nu_{\partial U^*}(x)\cdot e_n = 0$. The same computation\footnote{This follows by symmetry and because, denoting $\nu(x) = e_1$, the constant Neumann condition and harmonicity of $u$ impose $\partial^2_{x_1, x_i} u(x) = 0$ for all $i  \neq 1$.} as in \cite{Serrin71}  shows that   $D^2 v$  vanishes at $x$ as well. But then, the second-order Hopf boundary lemma, Lemma~\ref{lem:serrin}---applied to $v$ and the domain $U = U^*\cap \{x_n > t\}$---gives a contradiction.
\end{itemize}
\medskip  

    \noindent {\bf Step 2.} Conclusion.

    We split our analysis into several subdomains: 
    \begin{itemize}
        \item By Lemma~\ref{lem:gradn}, $\partial_{x_n} u>0$ in $\{u > 0\}\cap \{x_n\geq t/3\}\cap \{|x|\geq M\}$. 

        \item By the previous substep, we also know that $\partial_{x_n} u >0$ in $\{x_n\geq t\}\cap \overline{\{u>0\}}$. By continuity,   there is $\eps\in(0,2t/3]$ such that $\partial_{x_n} u>0$ in $\{x_n\geq t-\eps\}\cap \overline{\{u>0\}}\cap \{|x|\leq M\}$ as well. Together with the first bullet, we have $\partial_{x_n} u>0$ in $\{x_n\geq t-\eps\}\cap \overline{\{u>0\}}$. This ensures that $u\geq u_s^*$ in $\{s\leq x_n\leq s+\eps/2\}$ for every $s\in[t-\eps/4,t]$. Note that $s+\eps/2\geq t+\eps/4$ for this range of $s$.

        \item  By (b) in Lemma~\ref{lem:refln} we also have $u\geq u_s^*$ in $\{x_n\geq s\}\cap \{|x|\geq M_t\}$, for some $M_t \ge M$.

        \item Finally, there is $\tilde \eps\in (0, \eps)$ such that  $u\geq u_s^*$ in $\{x_n\geq s+\eps/2\}\cap \{|x|\leq M_t\}$ for all $s\in [t-\tilde\eps/4, t]$:
        
        By assumption, $u\geq u_t^*$ in $\{x_n\geq t\}$. Then, there is $\delta_\eps>0$ such that $u>u_t^*+\delta_\eps$ in $\overline{\{u_t^*>0\}}\cap \{x_n\geq t+\eps/4\}\cap\{|x|\leq M_t\}$: if equality were achieved at some $x$ with $u_t^*(x)>0$, this would contradict the interior maximum principle, and likewise if $x\in \FB(u_t^*)$ (and thus $x\in \FB(u)$ as well) by  Hopf's lemma.
        
        Hence,  $u>u_s^*+\delta_\eps/2$ in $\overline{\{u_s^*>0\}}\cap \{x_n\geq t+\eps/4\}\cap\{|x|\leq M_t\}$ for any $s\in[t-\tilde\eps/4,t]$ as well, for some $0<\tilde\eps \le \eps$ small,  by continuity and compactness. Since $s+\eps/2\geq t+\eps/4$ for this range of $s$, this confirms the claim opening this bullet. 
    \end{itemize}
    Putting all of the above together, we have shown that there is some  $\tilde\eps>0$ such that: For every $s\in [t-\tilde\eps/4,t]$, we have $\partial_{x_n} u > 0$ in $\{x_n\geq s\}\cap \overline{\{u>0\}}$ and $u\geq u_s^*$ in $\{x_n\geq s\}$. Hence $[t-\tilde\eps/4,t]\subset T$ as well. Since $t\in T$ with $t>0$ was arbitrary, this shows that $T$ is relatively open in $[0, \infty)$ and concludes the proof.
\end{proof}

\subsection{Non-compact free boundary case: symmetry in $x_1$}\label{sec:symx'}
 We assume the same setting as in the previous section.

\begin{lemma}\label{lem:expansion'}
    In the setting of Corollary~\ref{cor:errplus0}, $u$ is even in $x_n$ and up to a translation in the $x'=(x_1,...,x_{n-1})$ coordinates and a further rescaling\footnote{That is, considering $u(x'+x_0',x_n)$ for some appropriate $x_0'\in\R^{n-1}$ in place of $u$.} we have  
 \[
    |u -\tilde v| + |x| |\nabla(u -\tilde v)| \le |x|^{1-n-\alpha} \quad \mbox{in } \{u>0\}\cap \{x_n>0\} \setminus B_1,
    \]
    with
    \[
    \tilde v(x) := x_n - b + \frac{c}{|x- be_n|^{n-2}},
    \]
    for some $b, c > 0$, $|b|+|c|\le 1/2$.    
\end{lemma}

\begin{proof}
The evenness in $x_n$ follows from Proposition~\ref{prop:symxn}. Take $b$, $c = c_+$ and $d = d_+$ from Corollary~\ref{cor:errplus0}, with $c>0$ and $d_n=(n-2)bc$ by Lemma~\ref{lem:precisectt} and Corollary~\ref{cor:errplus0}. Set
\[
x_0':=\frac{d'}{(n-2)c}
\]
and replace $u(x',x_n)$ by $u(x'+x_0',x_n)$. Then, using
\[
|x+(x_0',0)|^{2-n}
=
|x|^{2-n}-(n-2)\frac{x_0'\cdot x'}{|x|^n}+O(|x|^{-n}),
\]
together with the analogous differentiated estimate, the resulting horizontal dipole term cancels that of $v_+$ in Corollary~\ref{cor:errplus0}: Outside of a large ball,
\[
 |u-\bigl(x_n-b+c|x|^{2-n}+d_nx_n|x|^{-n}\bigr)| 
 \le C|x|^{1-n-\alpha}.
\]
Notice that, since \(d_n=(n-2)bc\),
\[
c|x-be_n|^{2-n}
=
c|x|^{2-n}+d_n\,x_n|x|^{-n}+O(|x|^{-n}).
\]
 Therefore, also incorporating the differentiated quantities, 
\[
 |u-\tilde v|+|x|\,|\nabla(u-\tilde v)|
 \le C|x|^{1-n-\alpha},
\]
in $\{u > 0\}\cap \{x_n > 0\}$ outside of a large ball. Finally, after a harmless rescaling---already allowed in Corollary~\ref{cor:errplus0}---we may assume that the estimate holds in $\R^n\setminus B_1$ with constant $1$, and that \(|b|+|c|\le 1/2\). 
\end{proof} 

We put ourselves in the setting given by this lemma for the remainder of the section. We will now show:
\begin{proposition}\label{prop:symx'}
    In the setting of Lemma~\ref{lem:expansion'}, $u$ is even in $x_1$. Moreover,  $\partial_{x_1} u< 0$ in $\{x_1> 0\}\cap \overline{\{u>0\}}$.
\end{proposition}
As in the previous section, we start with some precise asymptotic properties.
\begin{lemma}[Gradient asymptotics]\label{lem:grad'}
For any $\delta > 0$, there is $M = M(\delta, n ) > 0$ such that: 
$$\partial_{x_1} u(x)<0\qquad\text{in}\quad \{x_1 \ge \delta/2\}\cap \{|x|\ge M\}\cap \{u > 0\}.$$ 
\end{lemma}
\begin{proof}
It follows from  Lemma~\ref{lem:expansion'} using
\[
\partial_{x_1} \tilde v(x)= 
\partial_{x_1}\left(\frac{c}{|x-b e_n|^{n-2}}\right)
= -(n-2)c\,\frac{x_1}{|x-b e_n|^n} <0
\]
in $\{x_1>0\}$.
\end{proof}

We write \(x=(x_1,x'')\), with \(x''=(x_2,\ldots,x_n)\in\R^{n-1}\). 
In this splitting we denote by
\[
B''_r(y''):=\{z''\in\R^{n-1}: |z''-y''|<r\},\qquad B''_r:=B''_r(0),
\]
the ball in the \(x''\)-variables. We define
\[
\sigma^\circ_t (x) := (2t-x_1, x'')\quad\mbox{and}\quad  x^\circ_t := \sigma^\circ_t (x),
\]
\[
u^\circ_t = u \circ \sigma^\circ_t \qquad \mbox{i.e.,}\quad  u_t^\circ(x_1, x'') := u(2t-x_1, x''). 
\]

\begin{lemma}[Horizontal initialization and reflection asymptotics] \label{lem:refln1}
In the setting of Lemma~\ref{lem:expansion'}, we have:

\noindent (a) There exists $t_0>8$ (large) such that for all $t>t_0$ we have:
$$
u\leq u_t^\circ \qquad\mbox{in}\quad\{x_1\geq t\} 
$$

\noindent (b) Given $\delta >0$, there is $R_\delta \ge 0$  such that, if $t\ge \delta$, 
$$
u\leq u_t^\circ \qquad\mbox{in}\quad\{x_1\geq t\}\setminus B_{R_\delta}.
$$
\end{lemma}

\begin{proof}
Given $x\in \R^n$ and  $s>0$, we  write 
\[
x_s := x-se_1. 
\]

Recall that 
\[
\tilde v(x+ be_n ) := x_n + \frac{c}{|x|^{n-2}}.
\]

On the one hand,  by the mean value theorem, given $x\in \{x_1\ge t\}$, $s\in (0, 2x_1)$, for any given $T_1$ and $T_2$ with 
\[
|x_1-s| \le T_1< T_2 \le x_1
\]  
there exists $\xi\in (T_1, T_2)$ such that 
\[
\begin{split}
\tilde v(x_s+ be_n)- \tilde v (x+ be_n) &= c((x_1-s)^2 +|x''|^2)^{1-n/2} - c((x_1)^2 +|x''|^2) ^{1-n/2}  \\
&\ge c(T_1^2 +|x''|^2)^{1-n/2} - c((T_2)^2 +|x''|^2)^{1-n/2} \\
&= c(n/2-1)2\xi (\xi^2 +|x''|^2)^{-n/2}(T_2-T_1).
\end{split}
\]
Hence, after a vertical translation (using $b < 1$):
\begin{equation}\label{ineq1}
\tilde v(x_s)- \tilde v (x) \ge \bar c \frac{T_1(T_2-T_1)}{(T_2 +|x''|+b)^{n}}\quad\text{in}\quad \{x_n > 0\}\setminus B_4. 
\end{equation}

On the other hand,   denoting $[x-se_1, x]$ the segment joining the points $x-se_1$ and $x$ for  $s>0$, the following implication holds:
\begin{equation}\label{impli11}
   [x_s, x]\subset \{u >0\}\cap\{x_n>0\} \setminus B_1  \Longrightarrow\   \big|(u -\tilde v)(x_s)  -(u -\tilde v)(x)\big| < |x''|^{-n-\alpha} s.  
\end{equation}
Indeed, it follows from the mean value theorem applied to $(u-\tilde v)$, using the gradient bound in Corollary~\ref{cor:errplus0}.

\medskip

\noindent {\bf Step 1.} We begin by showing (a). 

Let $M_1$ be given by Lemma~\ref{lem:grad'}, applied with $\delta=1$. Assume that $t>t_0$, with $t_0>M_1+1$ large to be chosen. We want to show that $u_t^\circ(x) \ge u(x)$ for all $x\in \{x_1 \ge t\}$.  If $x\in \{u=0\}$ there is nothing to show, so we may assume---taking also into account the symmetry in the $x_n$ direction---that 
\[
    x\in \{u>0\} \cap \{x_1 \ge t\}\cap\{x_n\ge 0\}=\Omega_+\cap \{x_1 \ge t\}.
\]

We now distinguish four cases: 
\begin{itemize}
    \item  If $ t<x_1<2t-1 $ and $x^\circ_t\in \R^n \setminus B_{t_0}$ we conclude $u(x_t^\circ)-u(x)>0$ by monotonicity, i.e. Lemma~\ref{lem:grad'}.

    \item  If $2t-1<x_1 \le  |x''|$ (so, in particular, $|x''|> 2t_0-1>M_1$), we claim that  $u(x_s)\ge u(x)>0$ for all $s\in (0, 2(x_1-t))$. First, if $s\in (0, x_1-1)$, then  $x_s\in \{y:y_1>1\}\setminus B_{M_1}$; by monotonicity (i.e. Lemma~\ref{lem:grad'}) we find that $u(x_s)\ge u(x)>0$ again. 
 
But then, from this point on, we can actually continue up to reaching the point $x^\circ_t$ (that is, $s=2(x_1-t)$) while keeping this relation. Indeed, \eqref{ineq1} for $T_1 = \max(x_1/2,|x_1-s|)$ and $T_2=x_1$  (using $x_1-s\le 1$ and hence $x_1- |x_1-s|\ge t$ then $T_2-T_1\ge t/2$) tells us:
\[
\tilde v(x_s)- \tilde v (x) \ge \bar c \frac{\max(x_1/2,|x_1-s|)t}{|x|^{n}}> 0.
\]
Then \eqref{impli11} gives:
\[
\begin{split}
    u(x_s)- u(x) > \tilde v(x_s)-\tilde v(x) -|x''|^{-\alpha-n} s \ge  \frac{\bar c x_1t/4}{|x|^{n}} - 2x_1|x''|^{-\alpha-n} .
\end{split}
\]
For $t\ge t_0$ (large enough), since $|x|$ and $|x''|$ are comparable by the assumption $x_1\leq |x''|$, the right hand-side above is positive, and we can continue until proving $u(x_t^\circ)- u(x) >0$.

\item If $x_1 \ge \max(2t-1, |x''|)$ and $x^\circ_t\notin B_{t_0}$: Assume first that additionally $x^\circ_t\in \{u > 0\}$; we will later reduce to this case. Then, by the triangle inequality  and using \eqref{ineq1} with $s=2(x_1-t)$, $|x_1-s|=|2t-x_1|$,  $T_1= \max(|2t-x_1|,|x''|-1)$, $T_2=T_1 +1 \le x_1$:
    \[
    \begin{split}
           u(x^\circ_t) &\ge  \tilde v(x^\circ_t) - \frac{1}{|x^\circ_t|^{n+\alpha-1}} 
           \\
           &\ge \tilde v(x) +  c_1 \frac{\max(|2t-x_1|,|x''|-1)}{(
        \max(|2t-x_1|+1,|x''|)+|x''|+b)^n}-|x_t^\circ|^{1-n-\alpha}
           \\
                      &\ge \tilde v(x) +  c_1 \frac{|x_t^\circ|}{(
        |2t-x_1|+1+|x''|+b)^n}-|x_t^\circ|^{1-n-\alpha}
           \\
           &\ge  u(x)   + c_2 |x^\circ_t|^{1-n} -|x^\circ_t|^{1-n-\alpha}-|x|^{1-n-\alpha}
           \\
           & \ge  u(x) + |x^\circ_t|^{1-n}(c_2 - 2|x^\circ_t|^{-\alpha})\\
           &> u(x)+ \frac{c_2}{2}|x^\circ_t|^{1-n}
    \end{split}
    \]
whenever $t_0$ is large enough, where we have also used  $|x_t^\circ|\le |x|$. In particular, $u(x^\circ_t)>u(x)$.

It remains to remove the additional assumption that $x^\circ_t\in \{u > 0\}$. Assume, for contradiction, that $x^\circ_t\in \{u = 0\}$ instead; sliding in the $e_n$ direction we find $s>0$ with $\sigma_t^\circ (x + se_n)\in\FB(u)$ and $(x + se_n)\in\{u>0\}$. Then, for every $h>0$ we have $\sigma_t^\circ (x + (s+h)e_n)\in\{u>0\}$ and $(x + (s+h)e_n)\in\{u>0\}$, thus we fall into one of the previous cases. In either case, this shows that $u(\sigma_t^\circ (x + (s+h)e_n))> u(x + (s+h)e_n)$, thus (sending $h\to 0$) we find $u(\sigma_t^\circ (x + se_n))\geq u(x + se_n)$. On the other hand $u(x + se_n)>0$ and $\sigma_t^\circ (x + se_n)\in\FB(u)$, reaching a contradiction.

\item Finally, if $x^\circ_t \in  B_{M_1}$ then $x\in B_{M_1}(2t e_1)$.
In this case, using Lemma~\ref{lem:barrier}, the interior maximum principle, and Hopf's lemma, there is $\eta > 0$ such that
\[
u \ge (|x_n|-b +\eta)_+\qquad\text{in}\quad B_{M_1} \quad\Longleftrightarrow \quad u_t^\circ \ge (|x_n|-b +\eta)_+\qquad\text{in}\quad B_{M_1}(2te_1).
\]

On the other hand, by the expansion in Lemma~\ref{lem:expansion'} we know that
\[
u(x+ 2te_1) \to (|x_n|-b)_+\quad\text{uniformly in } B_{M_1},\quad\text{as}\quad t\to\infty. 
\]
That is, for $t$ large enough, we have $u < |x_n|- b +\eta$ in $B_{M_1}(2te_1)$,
which shows $u_t^\circ \ge u$ in $B_{M_1}(2te_1)$. 
\end{itemize}

\medskip

\noindent {\bf Step 2.} Let us now show (b). We may assume that $t<t_0$ since otherwise we use case (a). 
We need to show that, given $\delta\in (0, t_0)$,  there is $R_\delta >t_0$ such that, for any $t \in (\delta, t_0)$, 
\[
u\leq u_t^\circ \qquad\mbox{in}\quad\{x_1\geq t\}\setminus {\mathcal C}_{R_\delta},\qquad\text{where}\quad {\mathcal C}_{R_\delta} :=  (t, 3R_\delta)\times B''_{R_\delta}
\]
As before, we may assume $x\in \{u>0\} \cap \{x_1 \ge t\} \cap \{x_n > 0\}\setminus {\mathcal C}_{R_\delta}  $. To prove that $u(x^\circ_t)>u(x)$ we argue as follows by separating between two regimes:
\begin{itemize}
\item If $x_1\le|x''|$,  we will show that $u(x_s)>u(x)>0$ for all $s\in (0,2(x_1-t))$. 
If $s\in (0,x_1-\delta/2)$ then $x_s \in  \{x_1 >\delta/2\}$ and we argue by monotonicity (provided we choose $R_\delta> M_{\delta/2}$). For $s\in (x_1-\delta/2, 2(x_1-t))$, instead, we use \eqref{ineq1} and \eqref{impli11}.

Arguing as in the second bullet in part (a), we have 
\[
u(x_s) - u(x) > \tilde v(x_s) - \tilde v(x) - |x''|^{-\alpha-n}s \ge c \frac{x_1(x_1-|x_1-s|)- 2x_1|x''|^{-\alpha}}{|x''|^n}.
\]
 Since $x_1-|x_1-s| > \delta/2$ and $|x''|\ge R_\delta$, for $R_\delta$ large enough we get the positive sign on the right-hand side, for all $s$ up to $2(x_1-t)$.

\item If $x_1>|x''|$,  since $t<R_\delta$ and $x\in \R^n  \setminus {\mathcal C}_{R_\delta}$, we have $x^\circ_t\in \R^n \setminus B_{R_\delta}$. Thus, if $u(x^\circ_t)>0$, using \eqref{ineq1} with $T_1= \max(|2t-x_1|,|x''|-\delta)$, $T_2=T_1 +\delta \le x_1$ as in the third bullet of part (a) we get
 \[
    \begin{split}
           u(x^\circ_t) &\ge \tilde v(x) +  c_1 \delta \frac{\max(|2t-x_1|,|x''|-\delta)}{(
        \max(|2t-x_1|+\delta,|x''|)+|x''|+b)^n}-|x_t^\circ|^{1-n-\alpha}
           \\
                      &\ge \tilde v(x) +  c_1\delta \frac{|x_t^\circ|}{(
        |2t-x_1|+\delta+|x''|+b)^n}-|x_t^\circ|^{1-n-\alpha}
           \\
           & \ge  u(x) + |x^\circ_t|^{1-n}(c_2 - 2|x^\circ_t|^{-\alpha})> u(x), 
    \end{split}
\]
since $|x_t^\circ|> R_\delta$, by taking $R_\delta$ large enough. The case $u(x_t^\circ) = 0$ can be argued as before. 
\end{itemize}

\end{proof}

\begin{proof}[Proof of Proposition~\ref{prop:symx'}] Let $T\subset [0, \infty)$ be defined as 
\[
T := \{t \in [0, \infty) : u \le u_t^\circ\ \  \text{in} \ \  \{x_1\ge t\}\quad \text{and}\quad \partial_{x_1} u < 0\ \  \text{in} \ \  \{x_1> t\}\cap\overline{\{u > 0\}}\, \}. 
\]
 We want to show that $0\in T$: This gives  $u\leq u_0^\circ$ and $\partial_{x_1} u < 0$ in $\{x_1 > 0\}\cap \overline{\{u > 0\}}$, and arguing symmetrically we deduce  $u=u_0^\circ$, i.e., $u$ is even in $x_1$, as we wanted. 

The proof proceeds now exactly as that  of Proposition~\ref{prop:symxn}, using Lemmas~\ref{lem:grad'} and \ref{lem:refln1} in place of Lemmas~\ref{lem:gradn} and \ref{lem:refln}. 
\end{proof}

\subsection{Proof of Theorems~\ref{thm:axymR3} and \ref{thm:axymRn}}
We can now prove the main theorems, by putting together the results in the article.
\begin{proof}[Proof of Theorem \ref{thm:axymRn}]
By Proposition~\ref{prop:FBdichotomy}, either $\FB(u)$ is bounded or there is some $R_0=R_0(u)$ such that $\FB(u)\cap \partial B_R\neq \emptyset$ for every $R\geq R_0$. In the first case, Proposition~\ref{prop:compactclass} gives the result. In the second one, we can put ourselves in the setting of Proposition~\ref{prop:FBdichotomy} and Corollary~\ref{cor:errplus0}. Then, Proposition~\ref{prop:symxn} shows that $u$ is even in $x_n$ and $\partial_{x_n} u>0$ in $\{x_n>0\}\cap \{u>0\}$. By Lemma~\ref{lem:expansion'}, performing a translation in the $x'=(x_1,...,x_{n-1})$ coordinates, and up to a further rescaling, it additionally satisfies the asymptotics indicated there.

To see that $u$ is radially symmetric in $x_1,...,x_{n-1}$, i.e. that $u(x^1)=u(x^2)$ if $x^1_n=x^2_n$, $|(x^1)'|=|(x^2)'|$ and $x^1\neq x^2$, let $e=\frac{x^2-x^1}{|x^2-x^1|}\in{\rm span}(e_1,...,e_{n-1})\cap \Sph^{n-1}$. The reflection with respect to $H_e=\{e\cdot y=0\}$ sends $x^1$ to $x^2$, thus it suffices to see that $u$ is even with respect to $H_e$. Let $\mathcal R:\R^{n-1}\to\R^{n-1}$ denote a rotation in the $x'$ coordinates which sends $e_1$ to $e$. The new function $\widetilde u(x)=u(\mathcal R x',x_n)$ is still a solution to the Bernoulli problem, even in $x_n$, satisfying the asymptotics in Lemma~\ref{lem:expansion'}. Then, Proposition~\ref{prop:symx'} shows that $\widetilde u$ is even with respect to $H_1=\{x_1=0\}$, which means that $u$ is even with respect to $H_e$. It moreover shows that $\partial_{x_1}\widetilde u<0$ in $\{x_1>0\}\cap\{u>0\}$, from which the negativity condition on $\partial_r u$ in Theorem~\ref{thm:axymRn} is seen to follow. Finally, combining the symmetries and monotonicity relations for $u$ with its asymptotic expansion in Lemma~\ref{lem:expansion'}, the graphical description for $\{u=0\}$---including the explicit asymptotics in Remark~\ref{rem:asympintro}---follows.
\end{proof}
The three-dimensional case follows unconditionally:
\begin{proof}[Proof of Theorem \ref{thm:axymR3}]
The result follows from Theorem~\ref{thm:axymRn} since $n=3$ is admissible there.
\end{proof}

\section{Total mean curvature}\label{sec:totHcurv}
\begin{definition}\label{def:totHcurv}
    Let $\Sigma^{n-1}\subset\R^n$ be an embedded hypersurface. We define the \textit{total mean curvature} of $\Sigma$ in an open subset $U\subset \R^n$ as
$$\mathbb H(\Sigma, U):=\int_{\Sigma\cap U} |{\rm H}_\Sigma|^{n-1}\, d\mathcal{H}^{n-1}.$$
We say that $\Sigma$ has \emph{finite total mean curvature} if $\mathbb H(\Sigma,\R^n)<\infty$.
\end{definition}
\begin{remark}
    Mean curvature plays a special role in the theory of the Bernoulli problem: Global solutions have mean-convex free boundaries (see, e.g., \cite[Lemma 3.2]{CFFS25}), and mean curvature is a main component in the stability inequality \eqref{eq:Q_nonnegative}. In fact, finite total (mean) curvature in two dimensions is precisely the condition used in \cite{BK24} to show that free boundaries of finite index solutions in $\R^2$ are finite unions of curves, reducing the classification result there to \cite{Tra14}.
\end{remark}
The goal of this section is to show that finite total mean curvature alone is a sufficient condition to obtain axial symmetry in all dimensions\footnote{In this way, total mean curvature seems to play a similar role to total curvature $\mathbb K(\Sigma,U):=\int_{\Sigma\cap U} |\mathrm{II}_{\Sigma}|^{n-1}\, d\mathcal{H}^{n-1}$ in minimal surface theory. Finite total curvature is known to imply finite index for global minimal hypersurfaces, as well as a strong regularity property for their ends, see \cite{Anderson84, Cho24Index}.}.
\begin{theorem}\label{thm:fincurv}
    Let $n\geq 3$, and let $u:\R^n\to[0,\infty)$ be a classical solution to the Bernoulli problem which is not one-dimensional. If $\FB(u)$ has finite total mean curvature, then the exact same conclusions as in Theorem~\ref{thm:axymRn} follow.
\end{theorem}    
The properties of $\mathbb H$ that we will need are precisely the following:
\begin{lemma}\label{lem:totHprop} The total mean curvature satisfies the following properties.
\begin{itemize}
    \item \textbf{Scaling invariance:} $\mathbb H(\lambda \Sigma, \lambda U)=\mathbb H(\Sigma, U)$ for any $\lambda>0$ and $U\subset\R^n$.
    \item \textbf{Monotonicity:} $\mathbb H(\Sigma, U)\leq \mathbb H(\Sigma, V)$ if $U\subset V\subset\R^n$.
    \item \textbf{Definiteness for the Bernoulli problem:} Let $U$ be a connected domain, and let $u:U\to[0,\infty)$ be a classical solution to the Bernoulli problem, with $\FB(u)\cap U\neq \emptyset$ and\footnote{The condition $|\nabla u|\leq 1$ is always true if $u$ is a restriction of  a global classical solution; see \cite[Lemma 3.2]{CFFS25}.} $|\nabla u|\leq 1$. If $\mathbb H(\FB(u), U)=0$, then $|D^2 u|\equiv 0$ in $\{u >0\}\cap U$.
\end{itemize}
\end{lemma}
\begin{proof}
    The first two properties are a standard consequence of the definition. The definiteness follows from \cite[Lemma 3.12]{CFFS25} and its proof, since it shows that ${\rm H}_{\FB(u)}$ can vanish at some $x_0\in \FB(u)\cap U$ only if there are some $e\in\Sph^{n-1}$ and $r>0$ such that $u\equiv ((x-x_0)\cdot e)_+$ in $B_r(x_0)$, and the result follows by unique continuation.            
\end{proof}
We now obtain a local regularity result under smallness of the total mean curvature, in the spirit of \cite{Anderson84}.
\begin{lemma}\label{lem:smallK_imp_reg}
Let $u:B_1\to [0,\infty)$ be a classical solution to the Bernoulli problem with $|\nabla u|\leq 1$ and $0\in \FB(u)$. Given $\eta>0$, there exists $\delta_0>0$ such that, if $\mathbb H(\FB(u),B_1)\leq \delta_0$,
then
$|D^2 u|\leq \eta$ in $\{u > 0\}\cap B_{1/2}$.
\end{lemma}
We first need the following lemma, which says that the regularity of the free boundary alone implies the regularity of $u$.
\begin{lemma}\label{lem:fbreg_imp_ureg}
Let $u:B_1\to [0,\infty)$ be a classical solution. If $|\mathrm{II}_{\FB(u)}|\leq C_0$ in $B_1$, then $|D^2u|+|D^3 u|\leq C$ in $\{u > 0\}\cap B_{1/2}$, for some $C = C(n, C_0)$.
\end{lemma}
\begin{proof}
By a scaling, translation, a finite covering argument, and interior harmonic estimates, it suffices to find some $r=r(C_0,n)$ such that $|D^2 u | + |D^3 u|\leq C$ in $U_r\cap B_{r/2}$, where $U_r$ is the connected component of $\{u>0\}\cap B_r$ with $0\in \overline{U_r}$.

Now, for $r \ll 1$ small enough we can ensure that $\FB(u)\cap \overline{U_r}$ is a union of at most two ``parallel'' graphs in $B_r$. Then, density estimates for $U_r$ (see \cite[Lemma 3.5]{CFFS25}) ensure that such graphs are at a uniform distance, so that up to making $r$ smaller (depending only on $C_0$ and $n$) we have that $\FB(u)\cap B_r$ consists of only one graphical component. Finally, a classical solution in $B_r$ whose free boundary is a single $C^2$ graph enjoys uniform $C^3$ estimates in $\{u > 0\}\cap B_{r/2}$, by e.g. \cite{Kinderlehrer-Nirenberg}.
\end{proof}
\begin{proof}[Proof of Lemma~\ref{lem:smallK_imp_reg}]
\noindent \textbf{Step 1.} We first show that $|\mathrm{II}_{\FB(u)}|\leq \eta$ in $B_{1/2}$ as long as $\delta_0>0$ is small enough.

We argue in the spirit of \cite{Anderson84}. Assume for contradiction that for some $\eta>0$ this were false. Then there would be $u_i$ as in the statement, with $\mathbb H(\FB(u_i),B_1)\xrightarrow{i\to\infty} 0$, but with $|\mathrm{II}_{\FB(u_i)}(x_i)|> \eta$ for some $x_i\in B_{3/4}$. In particular 
    $$\max_{y\in B_{4/5}} {\rm dist}(y, \partial B_{4/5})|\mathrm{II}_{\FB(u_i)}|\geq \frac{1}{20}\eta,$$
    and we can find a sequence $y_i$ with
    $$\max_{y\in B_{4/5}} {\rm dist}(y, \partial B_{4/5})|\mathrm{II}_{\FB(u_i)}|(y)={\rm dist}(y_i, \partial B_{4/5})|\mathrm{II}_{\FB(u_i)}|(y_i)=:r_i\geq \frac{1}{20}\eta.$$
    Put $\lambda_i:=|\mathrm{II}_{\FB(u_i)}|(y_i)$, and consider $\widetilde u_i=\lambda_iu_i(\lambda_i^{-1}x+y_i):B_{r_i/2}\to[0,\infty)$, which satisfies now $|\mathrm{II}_{\FB(\widetilde u_i)}|\leq 2$ and $|\mathrm{II}_{\FB(\widetilde u_i)}|(0)=1$. By Lemma~\ref{lem:fbreg_imp_ureg} and Arzel\`a--Ascoli, there is $r\in(0,\frac{1}{40}\eta)$ such that: A subsequence converges to a limit $u_\infty:B_r\to[0,\infty)$, and there is a connected component $U$ of $\{u_\infty>0\}\cap B_r$ such that $v:=u_\infty\, \mathbf{1}_U$ is a classical solution, $0\in \FB(v)$, and $|\mathrm{II}_{\FB(v)}|(0)=1$.
    
    On the other hand, by the first two bullets in Lemma~\ref{lem:totHprop}, $\mathbb H(\FB(v),B_r)=0$. By the third bullet there, then $|\mathrm{II}_{\FB(v)}|(0)=0$, which yields a contradiction.

\noindent \textbf{Step 2.} Conclusion.

By Step 1, Lemma~\ref{lem:fbreg_imp_ureg} and a simple covering argument, for $\delta_0>0$ small enough we have $|D^2u| + |D^3u| \leq C$ in $\{u > 0\}\cap B_{3/4}$ for some $C=C(n)$. The claim that---up to making $\delta_0>0$ smaller---we have $|D^2u| \leq \eta$ as well follows: Otherwise, the $C^3$ estimates and Arzel\`a--Ascoli allow us to pass to a solution with flat free boundary but nonzero Hessian, and the third bullet in Lemma~\ref{lem:totHprop} yields a contradiction again.
\end{proof}
We then obtain small asymptotic curvatures (with the correct scaling rate) for global solutions with finite total mean curvature:
\begin{lemma}\label{lem:vancurv}
    Let $n\geq 3$, and let $u$ be a classical solution to the Bernoulli problem in $\R^n$ with $\mathbb H(\FB(u),\R^n)<\infty$. Let $\eta>0$. Then, there is $R_0=R_0(u, \eta)$ such that $|D^2 u|\leq \frac{\eta}{|x|}$ in $\{u > 0\}\cap \R^n\setminus B_{R_0}$.
\end{lemma}
\begin{proof}
     Given $\delta_0>0$, there is $R_0=R_0(u, \delta_0)$ such that $\mathbb H(\FB(u),\R^n\setminus B_{R_0})\leq \delta_0$, just since $\mathbb H(\FB(u),\R^n)<\infty$. Moreover, $|\nabla u|\leq 1$ by \cite[Lemma 3.2]{CFFS25}. Applying Lemma~\ref{lem:smallK_imp_reg}---appropriately rescaled, on balls $B_{|x|/2}(x)$ for $x\in \R^n\setminus B_{2R_0}$---and making $R_0$ larger gives the result.
\end{proof}
We can then find the exact analogue of Lemma~\ref{lem:blowdown}:
\begin{lemma}[Compactness toward a two-plane configuration]\label{lem:blowdownH}
For every $\varepsilon\in(0,1)$ there exists $\delta=\delta(n,\varepsilon)>0$ such that the following holds.

Let $n\geq 3$, and let $u$ be a classical solution in $\R^n$. Assume that
$$\left[W(u,2)- W(u,1)\right]+\int_{\FB(u)\cap(\R^n\setminus B_\delta)}|{\rm H}_{\FB(u)}|^{n-1}\le \delta$$
and 
\[
\FB(u)\cap \bigl(B_2\setminus \overline{B}_1\bigr)\neq\varnothing.
\]
Then there exists a unit vector $e\in \mathbb S^{n-1}$ such that either
\begin{equation}\label{eq:7216tablok}
    \|u-(e\cdot x)_+\|_{L^\infty(B_{16})}\le \varepsilon \qquad \mbox{or} \qquad 
\|u-|e\cdot x|\|_{L^\infty(B_{16})}\le \varepsilon .
\end{equation}
Moreover, in the second case, letting $\mathcal R_e:\R^n\to\R^n$ be a rotation sending $e$ to $e_n$, we have:
\[
\mathcal R_e(\{u>0\}) = \{x_n > \tilde g^{(+)}(x')\} \cup \{x_n < \tilde g^{(-)}(x')\} \qquad\text{in}\quad B_8\setminus \overline{B}_{1/2},
\]
where $\tilde g^{(\pm)} : B_8' \to \mathbb{R}$,   $\tilde g^{(-)}< \tilde g^{(+)}$, and there is $C=C(n)>0$ such that
\[
\|\tilde g^{(\pm)}\|_{L^\infty(B_8')} + C^{-1}\|D^2 \tilde g^{(\pm)}\|_{L^\infty(B_8')} \leq \eps\leq 1
\]
for some $C$ depending only on $n$.

In particular,
\[
\FB(u)\cap \partial B_r\neq\varnothing
\qquad\text{for every }r\in(1,4).
\]
\end{lemma}
\begin{proof}
The proof starts exactly as that of Lemma~\ref{lem:blowdown}. The Hessian and nondegeneracy estimates \eqref{eq:89g4iubgilb0}--\eqref{eq:89g4iubgilb02} follow from Lemma~\ref{lem:vancurv} and a Harnack chain argument via interior harmonic estimates.

To see \eqref{eq:7216tablok} we argue by contradiction, considering $\delta_i=\frac{1}{i}$ and solutions $u_i$ contradicting the thesis, as in the proof of Lemma~\ref{lem:blowdown}. The corresponding subsequential limit is again a one-homogeneous solution. Moreover, by the vanishing curvature estimates from Lemma~\ref{lem:vancurv}---recall that we can make $\eta$ arbitrarily small there up to making $\delta_0$ smaller---we find that $|\mathrm{II}_{\FB(u_i)}|\to 0$ in $C^0_{loc}$. Then, by $C^2_{loc}(\R^n\setminus\{0\})$ convergence of the free boundaries (given by the $C^3$ estimates from Lemma~\ref{lem:fbreg_imp_ureg} plus Arzel\`a--Ascoli), we conclude that $\FB(u_\infty)=\{e\cdot x=0\}$ for some $e\in\Sph^{n-1}$---without ever requiring stability. If there is a tangential collapse of two sheets at some $x\in \FB(u_\infty)\setminus \{0\}$, we find again that $u_\infty=|e\cdot x|$. Otherwise, $u_\infty$ is a classical solution and thus of the form $u_\infty=(e\cdot x)_+$. In both cases we reach a contradiction, as eventually \eqref{eq:7216tablok} is satisfied by the $u_i$.

The graphicality follows by Lemma~\ref{cor_closetoV_disc_reg2}, just as in Lemma~\ref{lem:blowdown}, concluding the proof.
\end{proof}
Finally, we can give:
\begin{proof}[Proof of Theorem~\ref{thm:fincurv}]
With Lemma~\ref{lem:blowdownH} at hand, we can follow almost verbatim the arguments in the rest of the article. First, exactly as in Proposition~\ref{prop:FBdichotomy} it follows that either $\FB(u)$ is bounded, or there is some $R_0=R_0(u)$ such that $\FB(u)\cap \partial B_R\neq \emptyset$ for every $R\geq R_0$. In the first case, Proposition~\ref{prop:compactclass} gives radial symmetry. In the second one, we obtain the decomposition $u=u_++u_-$, away from a ball, where $u$ and the $u_\pm$ satisfy the asymptotics in Proposition~\ref{prop:FBdichotomy} and Corollary~\ref{cor:errplus0}: the same proofs go through, with the finite index condition replaced by having finite total mean curvature, since they only use the properties from Lemma~\ref{lem:blowdown} (which we can replace by Lemma~\ref{lem:blowdownH} in our case). Sections~\ref{sec:symxn} and \ref{sec:symx'} can then be replicated exactly, as they exclusively use the asymptotics in Corollary~\ref{cor:errplus0} plus the fact that $u$ is a global classical solution. We then conclude  by putting everything together as in the proof of Theorem~\ref{thm:axymRn}.
\end{proof}


\bibliographystyle{plain}
\bibliography{mybib.bib}

\end{document}